\documentclass[a4paper,11pt]{article}

\usepackage{amssymb}
\usepackage{amsthm}
\usepackage{amsmath}
\usepackage{amsfonts}
\usepackage{cite}
\usepackage{array}
\usepackage{graphicx}
\usepackage{epsfig}
\usepackage{float}
\usepackage{latexsym}
\usepackage{amsthm}
\usepackage{enumerate}
\usepackage{geometry}
\newtheorem{theorem}{Theorem}[section]

\newtheorem{lemma}{Lemma}[section]

\newtheorem{proposition}{Proposition}[section]
\newtheorem{assumption}{Assumption}[section]
\newtheorem{conjecture}{Conjecture}[section]

\newtheorem{property}{Property}[section]
\newtheorem{remark}{Remark}[section]
\newtheorem{delayedproof}{Proof}[section]
\let\oldremark\remark
\renewcommand{\remark}{\oldremark\normalfont}
\let\olddelayedproof\delayedproof
\renewcommand{\delayedproof}{\olddelayedproof\normalfont}

\renewcommand{\vec}[1]{{\normalfont\textbf{\text{#1}}}}
\newcommand{\bs}[1]{\boldsymbol{#1}}
\newcommand{\mat}[1]{\expandafter\MakeUppercase\expandafter{\text{#1}}}

\newcommand{\dx}{\mathrm{d}x}
\newcommand{\ds}{\mathrm{d}s}
\newcommand{\Real}{\mathbb{R}}
\newcommand{\domain}{\Omega}

\newcommand{\Tau}{\mathcal{T}}
\newcommand{\Chi}{\mathcal{X}}

\newcommand{\F}{\mathcal{F}}

\newcommand{\polyspace}{\mathbb{P}}
\newcommand{\V}{\mathcal{V}}
\newcommand{\Q}{\mathcal{Q}}
\newcommand{\W}{\mathcal{W}}
\newcommand{\C}{\mathcal{C}}
\newcommand{\B}{\mathcal{B}}

\renewcommand{\P}{\mathcal{P}}
\newcommand{\vareps}{\boldsymbol{\varepsilon}}

\newcommand{\bx}{B}
\newcommand{\face}{F}

\newcommand{\dmond}{D}
\renewcommand{\triangle}{T}
\newcommand{\vertex}{\vec{p}}

\newcommand{\faceCentre}{\vec{f}}

\newcommand{\pdd}[3]{\frac{\partial^{#2} #1}{\partial #3^{#2}}}

\newcommand{\norm}[1]{\left\lVert#1\right\rVert}
\newcommand{\abs}[1]{\left\vert#1\right\vert}
\newcommand{\seminorm}[1]{\abs{#1}}
\newcommand{\trinorm}[1]{
    {{\left\vert\kern-0.25ex\left\vert\kern-0.25ex\left\vert \left(#1\right)
    \right\vert\kern-0.25ex\right\vert\kern-0.25ex\right\vert}}
}
\newcommand{\coloneq}{{:{\kern-0.5ex =}}
}

\newcommand{\jump}[1]{\left[\!\left[ #1 \right]\!\right]}

\newcommand{\indicator}[1]{\Chi_{#1}}

\newcommand{\PtoBox}{\Pi_h}

\title{The Rhie-Chow stabilized Box Method for the Stokes problem}
\author{G. Negrini$^a$, N. Parolini$^a$ and M. Verani$^a$}

\begin{document}

\maketitle

\begin{center}
	{\small
		$^a$ MOX, Dipartimento di Matematica, Politecnico di Milano, Piazza Leonardo da Vinci 32, I-20133 Milano, Italy
	}
\end{center}

\pagenumbering{arabic}

\begin{abstract}
\noindent The Finite Volume method (FVM) is widely adopted in many different applications because of its built-in conservation properties, its ability to deal with arbitrary mesh and its computational efficiency. In this work, we consider the Rhie-Chow stabilized Box Method (RCBM) for the approximation of the Stokes problem. The Box Method (BM) is a piecewise linear Petrov-Galerkin formulation on the Voronoi dual mesh of a Delaunay triangulation, whereas the Rhie-Chow (RC) stabilization is a well known stabilization technique for FVM.
The first part of the paper provides a variational formulation of the RC stabilization and discusses the validity of crucial properties relevant for the well-posedeness and convergence of RCBM. Moreover, a numerical  exploration of the convergence properties of the method on 2D and 3D test cases is presented. The last part of the paper considers the  theoretically justification of the well-posedeness of RCBM and the experimentally observed  convergence rates. This latter justification hinges upon suitable assumptions, whose validity is numerically explored.
\end{abstract}

\section{Introduction}
The FVM is a popular numerical strategy for the spatial discretization of partial differential equations widely used for the solution of industrial flow problems. One crucial property of FVM is that, by construction, physical conservation laws governing in a given application are naturally discretized preserving global and local conservation properties. This makes the method very attractive when dealing with problems where conservation plays an important role, such as fluid mechanics and heat and mass transfer. This property is a consequence of the formulation of FVM. In fact, the core procedure of FVM is the imposition of the conservation law on each cell, or control volume, of the mesh \cite{LEVEQUE1992, MANGANI2015},  which is usually performed using the Gauss theorem and then numerically reconstructing fluxes through each face of the control volume. The conservation properties of the Finite Volume method give raise to robust numerical schemes that work on arbitrarily complex geometries \cite{HERBIN2000}.

In the present work we consider a particular formulation of the FVM, namely the so-called Box Method (BM) (also known as Finite Volume Element method or piecewise linear FVM). This method has been the object of an intense study in the literature. It was first introduced for scalar elliptic problems in \cite{BANKROSE1987, HACKBUSH1989} and, later, in \cite{EWINGLIN2002, XUZOU2009}. More recently, a stabilized version of BM has been applied to Stokes system in \cite{QUARTERONI2011}.

Two attractive features of BM are the simplicity of the formulation and an elegant relationship with the Finite Element method (FEM) (in this respect, see, e.g. \cite{BANKROSE1987, NEGRINI2021}), which hinges upon the fact that BM is the ``dual method'' of the finite element method,
i.e. it consists of a piecewise linear Petrov-Galerkin formulation on the Voronoi dual mesh of a Delaunay triangulation.

In this work, we introduce a variational formulation of the  stabilized version of BM to approximate the Stokes problem. This methods turns out to be equivalent to the discretization of the Stokes problem using piecewise linear elements for both velocity and pressure also employing numerical discretization of fluxes, where a suitable stabilization term is added in order to make the discrete problem well-posed. Concerning this latter aspect, we introduce a stabilization term equivalent to the so-called Rhie-Chow interpolation \cite{RHIECHOW1983, ZHANG2014, FERZIGER2020}, a common stabilization technique in Finite Volume applications. In the paper, we refer to the resulting stabilized method as to the Rhie-Chow Box Method (RCBM). We first present 2D and 3D numerical results to empirically explore the convergence properties of RBCM, conjecturing that the convergence rate equals the one of the piecewise linear FEM. The solution of these test cases is performed with OpenFOAM, an open-source software widely employed in both industrial and academic CFD applications.
Then we provide a detailed study of the continuity, consistency, coercivity and \textit{inf-sup} stability  properties of the RC stabilization. Despite RC stabilization is widely used in the OpenFOAM community, our results seem to be novel. The aforementioned properties are then employed to theoretically study the well-posedeness and convergence properties of RCBM. Also this second aspect seems to be novel in the literature.

More specifically,  the outline of the paper is as follows. In Section \ref{sec:stokes} we introduce the physical problem in the continuous variational framework. In Section \ref{sec:boxmethod} we first introduce the discrete functional setting and the construction of a Voronoi-dual mesh and then we introduce the general form of  stabilized Box Methods for the approximation of the Stokes problem. In Section \ref{sec:rhiechow} we introduce  the Rhie-Chow stabilization term  and discuss some crucial properties (consistency, continuity, coercivity and \textit{inf-sup} stability), which will be important to prove the well-posedeness and the convergence of RCBM. In Section \ref{sec:numerical} we numerically explore the convergence properties of RCBM. Finally, in Section \ref{sec:theory} we go through the theoretical analysis to prove the well-posedness and convergence of RCBM, which are obtained under suitable assumptions, whose validity is numerically explored.

\section{The Stokes problem} \label{sec:stokes}
Let $\domain \subset \mathbb{R}^d,\,d=2,3$ be a polyhedral bounded domain and let $\Gamma = \partial \domain$ be its boundary. We consider the steady incompressible Stokes problem for a Newtonia fluid:
\begin{equation}
\begin{aligned}
-\nu \Delta \vec{u} + \nabla p = \vec{f} &\quad \text{in }\domain, \\
\nabla \cdot \vec{u} = 0  &\quad \text{in }\domain,\\
\vec{u} = \vec{g} &\quad \text{on }\Gamma = \partial \domain
\end{aligned}
\label{eq:stokes}
\end{equation}
where $\vec{f} \in [L^2(\domain)]^d$ and $\vec{g} \in [H^{1/2}(\partial \domain)]^d$. Let $\V = H^1(\domain),\, \V_g = H^1_{\Gamma} = \{v \in \V:\, v=g \text{ on }\Gamma\}$ and $\Q=L^2(\domain)$ where $g$ is a sufficiently regular function. We will denote by a bold symbol the $d$-dimensional counterparts of those spaces, namely, $\bs{\V}=[H^1(\domain)]^d$, $\bs{\V}_{\vec{g}}=[H^1_{\Gamma}]^d$. Moreover, from now on, the standard norms for $H^1$ and $L^2$ spaces have to be intended on the whole domain $\domain$ where the domain of integration is not specified.

Then, we define the following bilinear forms:
\begin{equation}
\begin{aligned}
a:\bs{\V}\times\bs{\V} \rightarrow\mathbb{R}: \quad
& a(\vec{u},\vec{v})=\int_\domain \nu \nabla \vec{u} : \nabla \vec{v} \dx,
&& \quad \forall\vec{u},\vec{v} \in \bs{\V}, \\
b:\bs{\V}\times\Q \rightarrow\mathbb{R}: \quad
& b(\vec{v},p)=-\int_\domain \nabla \cdot \vec{v} \,p\dx,
&& \quad \forall\vec{v} \in \bs{\V}, \forall p \in \Q. \\
\end{aligned}
\end{equation}

The weak formulation of the problem reads: find $(\vec{u},p) \in \bs{\V} \times \Q,\, \vec{u}=\vec{g}\text{ on }\Gamma$, such that
\begin{equation}
\begin{aligned}
a(\vec{u},\vec{v}) + b(\vec{v},p)
&= (\vec{f},\vec{v})_\domain, \quad \forall \vec{v}\in \bs{\V}_0, \\
b(\vec{u},q) &=0, \quad \forall q \in \Q,
\end{aligned}
\label{eq:stokes:weak}
\end{equation}
where $(\cdot,\cdot)_\domain$ is the usual $L^2$ scalar product on $\domain$.

We also state the standard well-posedness result \cite{BABUSKA1971} for problem \eqref{eq:stokes:weak}:
\begin{theorem}[Well-posedness] \label{thm:stokes_well_posed_0}
The saddle-point problem \eqref{eq:stokes:weak} is well-posed if
\begin{enumerate}
\item the bilinear form $a$ is continuous and coercive;
\item the bilinear form $b$ is continuous;
\item the \textit{inf-sup} condition holds: $\exists\beta>0$ s.t.
\begin{equation}
\inf_{q\in \Q:q\neq0} \sup_{\vec{v}\in \bs{\V}:\vec{v}\neq\vec{0}}
\cfrac{b(\vec{v}, q)}{\norm{\nabla\vec{v}}_{L^2} \norm{q}_{L^2}}
\geq \beta > 0.
\label{eq:infsup}
\end{equation}
\end{enumerate}

Moreover, the solution $(\vec{u}, p) \in \bs{\V} \times \Q$ satisfies the following stability estimate:
\begin{equation}
\norm{\nabla \vec{u}}_{L^2(\domain)} + \norm{p}_{L^2(\domain)}
\leq C \left(\norm{\vec{f}}_{L^2(\domain)}
+ \norm{\vec{g}}_{H^{1/2}(\partial \domain)}\right).
\label{eq:stokes_stability}
\end{equation}

\end{theorem}

\section{The stabilized Box Method} \label{sec:boxmethod}
In this section we introduce the variational formulation of the stabilized Box Method  for the approximation of the Stokes problem (cf. problem \eqref{eq:stokes:box}). The formulation is obtained by employing suitable discrete variants (cf. equations \eqref{eq:stokes:bilforms:discrete})
of the bilinear forms appearing in the auxiliary problem \eqref{eq:stokes:continuous_box}.
Let $\Tau_h$ be a conforming and shape regular triangulation of $\domain$. We denote by $\triangle$ an element of $\Tau_h$ and by $h_\triangle$ the diameter of $\triangle \in \Tau_h$ and we set $h=\max_{\triangle\in\Tau_h} h_\triangle$. On $\Tau_h$ we define the following spaces of piecewise linear continuous functions:
\begin{equation*}
\begin{aligned}
\V_h &= \left\lbrace
    v_h \in C^0(\overline{\domain}):
    v_h\vert_\triangle \in \mathbb{P}^1(\triangle) \,\forall \triangle \in \Tau_h
\right\rbrace \subset \V, \\
\V_{h,g_h} &= \left\lbrace
    v_h \in \V_h:
    v_h = g_h \text{ on }\partial \domain
\right\rbrace \\
\end{aligned}
\end{equation*}
where $g_h$ is a suitable piecewise linear approximation of $g$ on $\partial \domain$. In this setting, for future use, we denote with $\seminorm{q}_{h,1}$ the broken $H^1$ norm on the elements of the triangulation for a function $q\in H^1_{loc}(\Omega)$.

\begin{figure}[ht]
\centering
\includegraphics[width=0.4\linewidth]{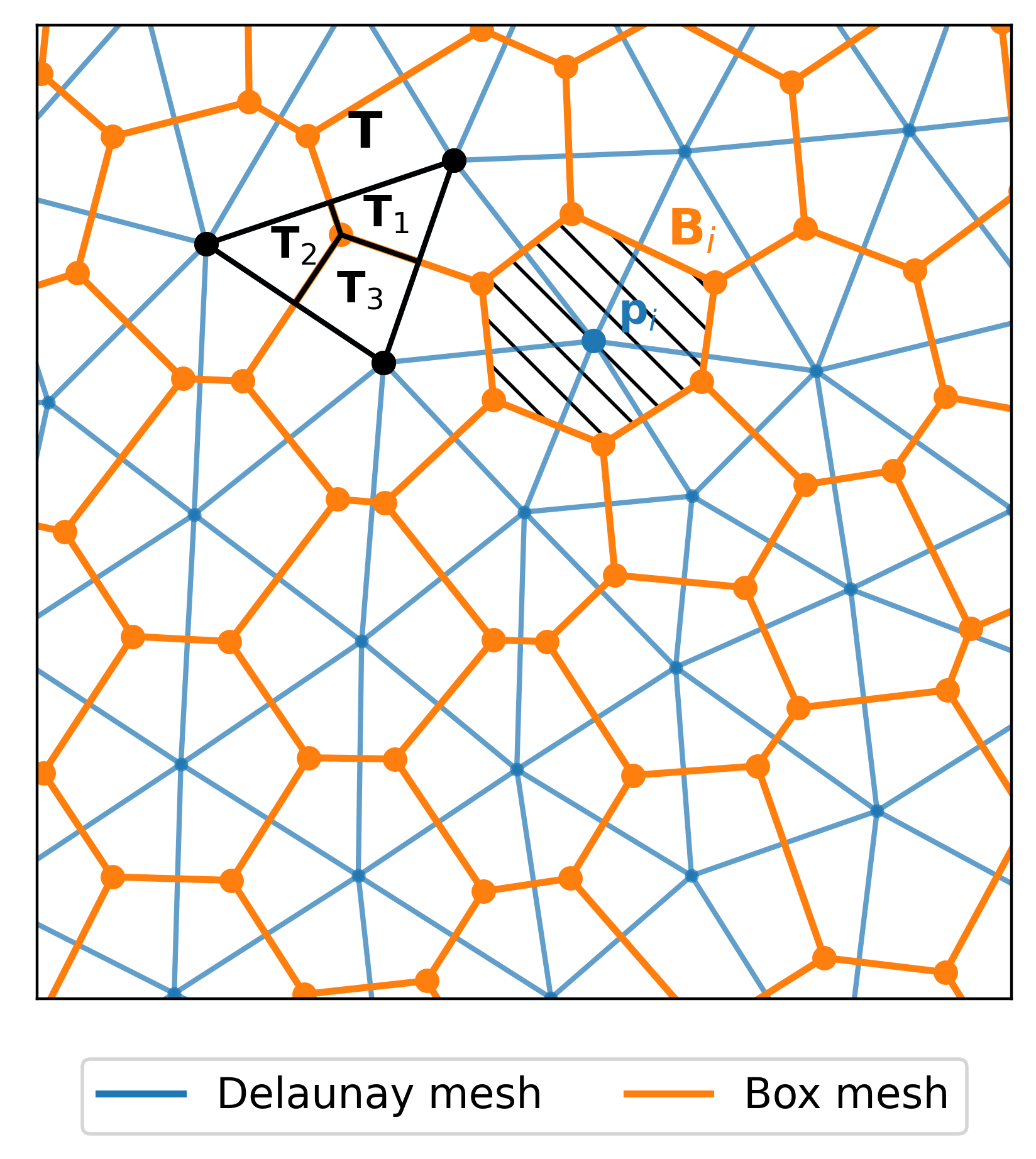}
\caption{Example of a Delaunay triangulation and its Voronoi dual mesh.}
\label{fig:tri_and_dual_mesh}
\end{figure}

We also define the ``box mesh'' (or dual mesh) $\B_h$ associated to $\Tau_h$. We introduce the set $\P_h=\{\vertex_i\}$ of vertices of $\Tau_h$ with $\P_h = \P_h^\partial \cup \P_h^o$, the set $\P_h^\partial$ containing the boundary vertices of $\Tau_h$ and the set $\P_h^o$ containing the interior vertices of $\Tau_h$. We denote by $\P_{\vertex_i}$ the set of triangles sharing vertex $\vertex_i$. Let then $\B_h=\{\bx_i\}_{\vertex_i \in \P_h^o}$ be the set of boxes $\bx_i$. Each box is a polyhedron with a skeleton consisting of straight lines connecting the circumcentres of each element $\triangle \in P_{\vertex_i}$ (see Figure \ref{fig:tri_and_dual_mesh}) and outer unit normal vector $\vec{n}_i$.

\begin{figure}[ht]
\centering
\includegraphics[width=0.6\linewidth]{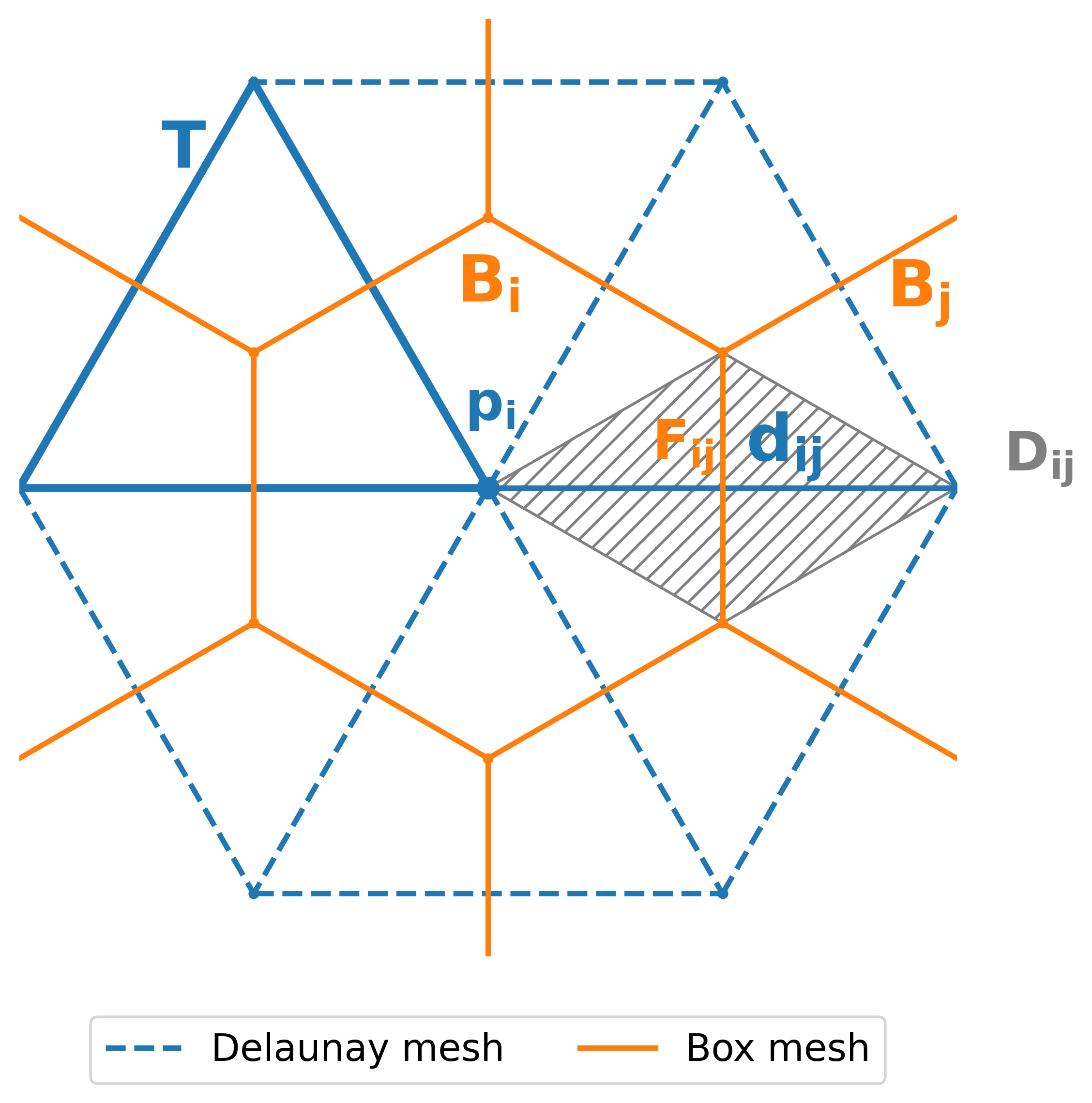}
\caption{Scheme of dual mesh geometrical quantities.}
\label{fig:scheme_dual_geom}
\end{figure}

We also introduce some mesh quantities (Figure \ref{fig:scheme_dual_geom}) that will be instrumental to define the discrete bilinear forms (see equations \eqref{eq:stokes:bilforms:discrete}). We denote by $N_{\bx_i}$ the number of faces of box $\bx_i$ and by $N_\bx=\max_{i}N_{\bx_i}$. Considering a box $\bx_i$, we denote by $\bx_{j}$ the box that shares the face $\face_{ij}$ with $\bx_i$ and by $\F_h = \{\face_{ij}\}$ the set of all faces. Let $d_{ij}$ be the distance between the barycentres of boxes $\bx_i$ and $\bx_{j}$, $\vec{n}_{ij}$ the unit normal vector directed outwards of $\bx_i$ and $w_{ij}$ be the interpolation weight (the ration between distance of $\bx_{j}$ barycentre from face $\face_{ij}$ and $d_{ij}$). Let $\dmond_{ij}$ be the ``diamond'', i.e. the polyhedron whose nodes are the barycentres and the common nodes of two boxes (shaded region in Figure \ref{fig:scheme_dual_geom}); notice that $\abs{\dmond_{ij}}=\abs{\face_{ij}}d_{ij}/d$. Moreover, let $G_i=\{\bx_j:\,\exists \bx_i \cap\bx_j=\face_{ij},\,\text{for some }\face_{ij}\in\F_h\}$ be the set of boxes that have a face in common with $\bx_i$. We also denote by $\phi_i=\phi\vert_{\bx_i}$ the restriction of a function $\phi_h\in{\V_h}$ evaluated on box $\bx_i$.

The following observation will be useful in the sequel of the paper.
\begin{property} \label{property:constant_on_diamond}
Let us consider a function $p_h \in \V_h$  and recall from Figure \ref{fig:scheme_dual_geom} the meaning of the segment $d_{ij}$, the face $\face_{ij}$ and the diamond $D_{ij}$. Moreover, we introduce the subset $\Tau_h^{ij}=\{\triangle\in\Tau_h:\triangle\cap\face_{ij}\neq\emptyset\}$ of elements of the primal mesh.
By the fact that $p_h \in C^0(\Omega)$, we have
\[\left(\nabla p_h\vert_{\triangle}\cdot\vec{n}_{ij}\right)\vert_{d_{ij}}=c_{ij},\]
  for every $\triangle\in\Tau_h^{ij}$, where $c_{ij}\in\mathbb{R}$ is a constant value, whose value varies with $F_{ij}$. Now, observing that $\vec{n}_{ij}$ is constant over the face $\face_{ij}$ and it is aligned with $d_{ij}$ yields $(\nabla p_h\vert_{D_{ij}}\cdot\vec{n}_{ij})\vert_{F_{ij}}=c_{ij}$. Thus, from now on,  $\vec{n}_{ij}$ or $\nabla p_h\cdot\vec{n}_{ij}$ appearing in a volume integral on the diamond $D_{ij}$, are meant as constant extensions over $D_{ij}$.

In view of the above discussion, the following equalities, that will be repeatedly employed in the sequel of the paper, hold:
\begin{equation*}
\begin{aligned}
\int_{\face_{ij}}\nabla p_h\cdot\vec{n}_{ij}\ds =&
\abs{\face_{ij}}\nabla p_h\cdot\vec{n}_{ij}
= \cfrac{\abs{\face_{ij}}}{\abs{\dmond_{ij}}}\int_{\dmond_{ij}}
\nabla p_h\cdot\vec{n}_{ij} \dx \\
= &\cfrac{1}{\abs{\dmond_{ij}}}\int_{\face_{ij}}\int_{\dmond_{ij}}
\nabla p_h\cdot\vec{n}_{ij} \dx\ds.
\end{aligned}
\end{equation*}

\end{property}

In the sequel we will work under the following assumption on the regularity of the computational mesh.
\begin{assumption}[Mesh regularity] \label{ass:mesh_regularity}
Let $\B_h$ be the Voronoi-type dual mesh of a Delaunay triangulation $\Tau_h$ such that
\begin{equation*}
\exists \delta > 0: h_m = \delta h,\qquad h_m= \min_{\triangle \in \Tau_h} h_\triangle.
\end{equation*}
Moreover, we assume that mesh size does not change too much between neighbouring boxes. Hence $\forall \bx_i \in \B_h$ and $\forall \face_{ij} \in \F_h$, for some $\triangle: \triangle\cap \bx_i \neq \emptyset$,
\begin{equation*}
d_{ij}\simeq h_\triangle,\quad \abs{\face_{ij}}\simeq h_\triangle^{d-1}, \quad
\abs{\bx_i} \simeq h_\triangle^d,\quad.
\end{equation*}
Moreover, $\exists C_1,C_2 >0:\, C_1h\le h_\triangle \le C_2 h, \forall \triangle \in \Tau_h$.
\end{assumption}

\begin{remark}
The primal mesh $\Tau_h$ is assumed to be a Delaunay triangulation, i.e. no vertex of the triangulation is inside the circumcircle of any triangle of $\Tau_h$. Under Assumption \ref{ass:mesh_regularity}, the dual mesh, defined as above, will be a Voronoi-type dual mesh and so it will be orthogonal, i.e. segments connecting two barycentres of adjacent boxes are parallel to the unit normal vector of the face between them. Under these assumptions, the Box method and the classical FVM will present similar features.
\end{remark}
\medskip

On $\B_h$ we introduce the space of piecewise constant functions.
\begin{equation*}
\begin{aligned}
\W_h &= \left\lbrace
    w_h \in L^2(\domain):
    w_h\vert_{\bx_i} \in \polyspace^0(\bx_i), \,\forall \bx_i \in \B_h
\right\rbrace, \\
\end{aligned}
\end{equation*}
where the relation between the trial and test spaces is defined using the following lumping map: let $v_h \in \V_h,$
\begin{equation}
\PtoBox:\V_h \rightarrow \W_h: \quad
v_h = \sum_{\bx_i \in \B_h}
v_h(\vertex_i)\varphi_i
\mapsto \PtoBox v_h = \sum_{\bx_i \in \B_h}
v_h(\vertex_i) \chi_i,
\label{eq:lumping_map}
\end{equation}
where $\varphi_i$ and $\chi_i$ are the piecewise linear shape functions and the characteristic functions of the boxes, respectively. Moreover, for the lumping map, we define the following notation, which will be intensively used in the next sections,
\begin{equation}
\PtoBox\phi_i=\PtoBox\phi_h\vert_{\bx_i},\quad \forall \phi_h\in{\V_h}.
\label{eq:notation_ptobox}
\end{equation}

We define a discrete $H^1$-norm that uses the normal gradient to each face of the box mesh (c.f. proof \ref{proof:star_norm}), namely
\begin{equation}
\seminorm{q}_* =\left(
\sum_{\face_{ij} \in \F_h}d_{ij}\int_{\face_{ij}}
\abs{\pdd{q}{}{\vec{n}_{ij}}}^2\ds
\right)^{\frac{1}{2}}, \qquad q \in H^1(\domain).
\label{eq:starnorm}
\end{equation}

\begin{proposition}[$*$-norm properties] \label{prop:starnorm_properties}
The following properties hold: $\forall q_h \in \V_h$,
\begin{equation*}
\begin{aligned}
\norm{\PtoBox q_h}_{L^2} \leq& C \seminorm{q_h}_*, \\
\seminorm{q_h}_*\leq& C \seminorm{q_h}_{H^1}, \\
\seminorm{q_h}_{H^1} \leq& C \seminorm{q_h}_*, \\
\norm{q_h}_{L^2} \leq& \seminorm{q_h}_*, \\
\seminorm{q_h}_* \leq& 2h_m^{-1}\norm{\PtoBox q_h}_{L^2}. \\
\end{aligned}
\end{equation*}
\end{proposition}
In particular, the $*$-norm is equivalent to the $H^1$-seminorm on $\V_h$.

Moreover we define the following mesh dependent norm:
\begin{equation}
\trinorm{\vec{v}_h, q_h}_{box} = \left(
\seminorm{\vec{v}_h}_{H^1}^2 +
\norm{\PtoBox q_h}_{L^2}^2 + \seminorm{q_h}_{\triangle,*}^2
\right)^\frac{1}{2},
\label{eq_norm}
\end{equation}
$\forall (\vec{v}_h,q_h)\in \bs{\V}_h \times \V_h$, where, for $q \in H^1(\domain)$,
\begin{equation}
\seminorm{q}_{\triangle,*} = \left(
\sum_{\face_{ij} \in \F_h}h_\triangle^3d_{ij}\int_{\face_{ij}}
\abs{\pdd{q}{}{\vec{n}_{ij}}}^2\ds
\right)^{\frac{1}{2}},
\label{eq:starnorm_mesh_dep}
\end{equation}
is a variant of the $*$-norm defined in Proposition \ref{eq:starnorm}.

Employing piecewise constant functions on $\B_h$, we introduce the following {\it Auxiliary Box Method} for problem \eqref{eq:stokes}: find $(\vec{u}_A, p_A) \in \bs{\V}_{h,\vec{g}_h} \times \V_h$, such that
\begin{equation}
\begin{aligned}
a_B(\vec{u}_A,\PtoBox\vec{v}_h) + b_B(\PtoBox\vec{v}_h,p_A)
&= (\vec{f},\PtoBox\vec{v}_h)_\domain,
&&\quad \forall \vec{v}_h\in \bs{\V}_h, \\
c_B(\vec{u}_A,\PtoBox q_h) + s_B(p_A,\PtoBox q_h)
&= 0,
&&\quad \forall q_h \in \V_h
\end{aligned}
\label{eq:stokes:continuous_box}
\end{equation}
where
\begin{equation}
\begin{aligned}
a_B:\bs{\V}_h\times\bs{\W}_h \rightarrow \mathbb{R}: \quad
& a_B(\vec{v}_h,\vec{w}_h)=
-\sum_{\bx_i \in \B_h}
\int_{\partial \bx_i}
\nu\pdd{\vec{v}_h}{}{\vec{n}_i} \cdot \vec{w}_h \ds, \\
b_B:\bs{\W}_h\times\V_h \rightarrow \mathbb{R}: \quad
& b_B(\vec{w}_h,q_h)=
\sum_{\bx_i \in \B_h}
\int_{\partial \bx_i}
q_h\vec{n}_i\cdot \vec{w}_h\ds, \\
c_B:\bs{\V}_h\times\W_h \rightarrow \mathbb{R}: \quad
& c_B(\vec{v}_h, w_h)=
\sum_{\bx_i \in \B_h}
\int_{\partial \bx_i}
\vec{n}_i \cdot \vec{v}_h w_h \ds, \\
s_B:\V_h\times\W_h \rightarrow \mathbb{R}, \quad &
\end{aligned}
\label{eq:stokes:bilforms}
\end{equation}
where $s_B$ is a suitable stabilization term (that will be defined in section \ref{sec:rhiechow}). Let us define the compact form
\begin{equation}
\C_B((\vec{u}_A,p_A),(\PtoBox \vec{v}_h,\PtoBox q_h)) =
a_B(\vec{u}_A,\PtoBox \vec{v}_h) +
b_B(\PtoBox\vec{v}_h,p_A) +
c_B(\vec{u}_A,\PtoBox q_h) +
s_B(p_A,\PtoBox q_h).
\label{eq:quadriform_cont}
\end{equation}
Then, the above problem can be written as
\begin{equation}
\C_B((\vec{u}_A,p_A),(\PtoBox \vec{v}_h,\PtoBox q_h))
= (\vec{f},\PtoBox \vec{v}_h)_\domain,
\end{equation}
$\forall (\vec{v}_h,q_h)\in\bs{\V}_h\times\V_h$.

\begin{remark}
Let $a$ and $b$ be the bilinear forms of the piecewise linear FEM formulation:
\begin{equation}
\begin{aligned}
a(\vec{u}_h,\vec{v}_h) + b(\vec{v}_h,p_h)
&= (\vec{f},\vec{v}_h)_\domain, &&\quad \forall \vec{v}_h\in \bs{\V}_{h,0}, \\
-b(\vec{u}_h,q_h)
+ s(p_h,q_h)
&= 0, &&\quad \forall q_h \in \V_h,
\end{aligned}
\label{eq:bilforms_fem}
\end{equation}
where $s$ is an opportune stabilization term (e.g. Brezzi-Pitk{\"a}ranta \cite{BREZZI1984} or Interior Penalty \cite{BURMAN2006}). Then, it can be proven \cite{QUARTERONI2011} that the BM bilinear forms introduced in equation \eqref{eq:stokes:bilforms} have a strict relationship with the FEM ones, namely:
\begin{equation}
\begin{aligned}
a_B(\vec{u}_h,\PtoBox\vec{v}_h) =& a(\vec{u}_h,\vec{v}_h), \\
b_B(\PtoBox\vec{v}_h,p_h) =& b(\vec{v}_h,p_h), \\
c_B(\vec{u}_h,\PtoBox q_h) =& -b(\vec{u}_h,q_h).
\end{aligned}
\label{eq:bilforms_equivalence}
\end{equation}
Moreover, due to this fact, many properties such as coercivity and consistency are preserved between the two methods. For the proof of equalities \eqref{eq:bilforms_equivalence}, we refer to Lemma \ref{lemma:bilforms_equivalence}.
\end{remark}
\medskip

We then introduce the following discrete bilinear forms:
\begin{subequations}
\begin{align}
\tilde{a}_B(\vec{v}_h,\vec{w}_h)
\coloneq&-\sum_{\bx_i \in \B_h}\sum_{\bx_j\in G_i} \nu\cfrac{\abs{\face_{ij}}}{d_{ij}}
\left[\PtoBox \vec{v}_j - \PtoBox\vec{v}_i\right]\cdot \vec{w}_h,
\label{eq:stokes:bilforms:discrete:a}\\
\tilde{b}_B(\vec{w}_h,q_h)
\coloneq&\sum_{\bx_i \in \B_h}\sum_{\bx_j\in G_i}
\abs{\face_{ij}}
\left[w_{ij}\PtoBox q_i + (1-w_{ij})\PtoBox q_j\right]
\vec{n}_{ij}\cdot \vec{w}_h, \\
\tilde{c}_B(\vec{v}_h, w_h) \coloneq&
\sum_{\bx_i \in \B_h}\sum_{\bx_j\in G_i} \abs{\face_{ij}}
\left[w_{ij}\PtoBox \vec{v}_i + (1-w_{ij})\PtoBox \vec{v}_j\right]\cdot
\vec{n}_{ij} w_h.
\end{align}
\label{eq:stokes:bilforms:discrete}
\end{subequations}

Here, the Laplacian operator is discretized using a finite difference between barycentres of two adjacent boxes, while the gradient and the divergence are discretized using a linear interpolation between the same two values using as weights the distances of the barycentres with respect to the face centre. By the fact that the discrete functions are piecewise linear on the primal mesh and given the orthogonality of the dual mesh, we have $a_B=\tilde{a}_B$. On the other hand, the bilinear forms $\tilde{b}_B$ and $\tilde{c}_B$ are not exactly equal to $b_B$ and $c_B$, respectively, and this will be taken into account in the analysis.

Finally, adopting the discrete bilinear forms \eqref{eq:stokes:bilforms:discrete} in the Auxiliary Box Method \eqref{eq:stokes:continuous_box}, we obtain the following scheme, that, from now on, we refer as to the {\it Box Method}: find $(\vec{u}_B, p_B) \in \bs{\V}_{h,\vec{g}_h} \times \V_h$, such that
\begin{equation}
\begin{aligned}
a_B(\vec{u}_B,\PtoBox\vec{v}_h) + \tilde{b}_B(\PtoBox\vec{v}_h,p_B)
&= (\vec{f},\PtoBox\vec{v}_h)_\domain,
&&\quad \forall \vec{v}_h\in \bs{\V}_h, \\
\tilde{c}_B(\vec{u}_B,\PtoBox q_h) + \tilde{s}_B(p_B,\PtoBox q_h)
&= 0,
&&\quad \forall q_h \in \V_h,
\end{aligned}
\label{eq:stokes:box}
\end{equation}
where we employed an approximate form of the stabilization term $\tilde{s}_B$.
Introducing the compact form
\begin{equation}
\widetilde{\C}_B((\vec{u}_B,p_B),(\PtoBox \vec{v}_h,\PtoBox q_h)) =
a_B(\vec{u}_B,\PtoBox \vec{v}_h) +
\tilde{b}_B(\PtoBox\vec{v}_h,p_B) +
\tilde{c}_B(\vec{u}_B,\PtoBox q_h) +
\tilde{s}_B(p_B,\PtoBox q_h),
\label{eq:quadriform}
\end{equation}
the problem above can be rewritten as
\begin{equation}
\widetilde{\C}_B((\vec{u}_B,p_B),(\PtoBox \vec{v}_h,\PtoBox q_h)) =
(\vec{f},\PtoBox \vec{v}_h)
\end{equation}
$\forall (\vec{v}_h,q_h)\in\bs{\V}_h\times\V_h$.

\section{The Rhie-Chow stabilization} \label{sec:rhiechow}
In this section we introduce the Rhie-Chow (RC) stabilization to be employed in the stabilized Box Method  \eqref{eq:stokes:box} giving rise to the Rhie-Chow stabilized Box Method (RCBM). Moreover, we study crucial properties of the RC stabilization (consistency, continuity, coercivity, \textit{inf-sup} stability) that will be important for the well-posedeness and  convergence of RCBM.

Let $N$ be the number of vertices $\vertex$, $N_u=d N$ the number of d.o.f of velocity and $N_p=N$ the number of d.o.f. of pressure,
\begin{equation}
\begin{aligned}
\bs{\V}_h: \span\left\lbrace \bs{\phi}_j \right\rbrace_{j=1}^{N_u},
\qquad \V_h: \span\left\lbrace \phi_j \right\rbrace_{j=1}^{N_p}.
%&= \left\lbrace
%\left[\begin{matrix}\phi_1\\0\\0\end{matrix}\right],\dots,
%\left[\begin{matrix}\phi_{N}\\0\\0\end{matrix}\right],
%\left[\begin{matrix}0\\\phi_{N + 1}\\0\end{matrix}\right],\dots,
%\left[\begin{matrix}0\\\phi_{2N}\\0\end{matrix}\right],
%\left[\begin{matrix}0\\0\\\phi_{2N + 1}\end{matrix}\right],\dots,
%\left[\begin{matrix}0\\0\\\phi_{3N}\end{matrix}\right]
%\right\rbrace, \\
%\V_h: \span\left\lbrace \phi_j \right\rbrace_{j=1}^{N_p}& \\
\end{aligned}
\label{eq:discrete_spaces_basis}
\end{equation}
For space $\bs{\W}_h,\W_h$ the basis are the projections on boxes of the ones of $\bs{\V}_h,\V_h$.

Then, we write $\vec{u}_B$ and $p_B$ as linear combinations of the basis functions:
\begin{equation}
\vec{u}_B(\vec{x}) = \sum_{j=1}^{N_u} u_j \bs{\phi}_j(\vec{x}) ,\qquad
p_B(\vec{x}) = \sum_{j=1}^{N_p} p_j \phi_j(\vec{x}).
\label{eq:solution_linear_combination}
\end{equation}

To find the expressions of the matrices of the problem, A, B and $\mat{B}^\intercal$, we plug the expressions for $\vec{u}_B$ and $p_B$ into the bilinear forms setting, $\vec{w}_h=\PtoBox\bs{\phi}_i$ and $w_h=\PtoBox\phi_i$, where the operator $\PtoBox$ acts component-wise.
\begin{subequations}\label{eq:explicit_matrices}
\begin{align}
(\text{A}\vec{u})_i&
\begin{aligned}[t]=&
\sum_j \text{A}_{ij}\vec{u}_j = \tilde{a}_B(\vec{u}_B,\vec{w}_h)
=\sum_{j=1}^{N_u} u_j \tilde{a}_B(\bs{\phi}_j, \PtoBox \bs{\phi}_i)\\
=& -\sum_{j=1}^{N_u} u_j\sum_{\bx_k\in G_j}
\nu \cfrac{\abs{\face_{jk}}}{d_{jk}}\left[\PtoBox \bs{\phi}_k - \PtoBox\bs{\phi}_j\right]
\cdot \PtoBox\bs{\phi}_{i}\\
=& u_i \sum_{\bx_j\in G_i}\nu \cfrac{\abs{\face_{ij}}}{d_{ij}}
-\sum_{\bx_j\in G_i} u_j
\nu\cfrac{\abs{\face_{ij}}}{d_{ij}},
\end{aligned}\\
%%%%%%%%%%%%%%%%%%%%%%%
(\text{B}^\intercal \vec{p})_i  &\begin{aligned}[t]
 =& \sum_j\text{B}^\intercal_{ij} \vec{p}_j =
\tilde{b}_B(\vec{w}_h,p_B)
=\sum_{j=1}^{N} p_j \tilde{b}_B(\PtoBox\bs{\phi}_i,\phi_j)\\
=&\sum_{j=1}^{N} p_j
\sum_{\bx_k\in G_j} \abs{\face_{jk}}
\left[w_{ij}\PtoBox \phi_j + (1-w_{ij})\PtoBox \phi_k\right]
\vec{n}_{jk}\cdot \bs{\phi}_i \\
=& p_i\sum_{\bx_j\in G_i} \abs{\face_{ij}}
w_{ij}\vec{n}_{ij} \cdot \vec{e}_i
+ \sum_{\bx_j\in G_i}
p_j \abs{\face_{ij}}(1-w_{ij})\vec{n}_{ij}\cdot \vec{e}_i,
\end{aligned}\\
%%%%%%%%%%%%%%%%%%%%%%%%%%
%\displaybreak\nonumber\\
(\text{B}\vec{u})_i  &\begin{aligned}[t]
=& \sum_j \text{B}_{ij}\vec{u}_j =
\tilde{c}_B(\vec{u}_B,w_h)
= \sum_{j=1}^{3N} u_j \tilde{c}_B(\bs{\phi}_j, \PtoBox \phi_i) \\
=&\sum_{j=1}^{3N} u_j
\sum_{\bx_k\in G_j} \abs{\face_{jk}}
\left[w_{ij}\PtoBox \bs{\phi}_j + (1-w_{ij})\PtoBox \bs{\phi}_k\right]
\cdot \vec{n}_{jk}\PtoBox\phi_i\\
=& \sum_{\substack{i=1+(l-1)N,\\\,l=1,\dots,d}}^{lN} u_i \sum_{\bx_j\in G_i}\abs{\face_{ij}}w_{ij} \vec{e}_l \cdot \vec{n}_{ij}
+ \sum_{\bx_j\in G_i}
u_j \abs{\face_{ij}}(1-w_{ij}) \vec{e}_l \cdot \vec{n}_{ij}.
\end{aligned}
\end{align}
\end{subequations}

Using equations \eqref{eq:explicit_matrices}, the algebraic linear system associated to problem \eqref{eq:stokes:box}, considering $s_B=0$, reads:
\begin{equation}
\left[
\begin{matrix} \text{A} & \text{B}^\intercal \\ \text{B} & 0 \\ \end{matrix}
\right]
\left[\begin{matrix} \vec{u} \\ \vec{p} \end{matrix} \right]
=  \left[\begin{matrix} \vec{F} \\ \vec{0} \end{matrix} \right]
\label{eq:monolithic_alge}
\end{equation}
where $\vec{F}$ is the discretization of the right-hand-side.

For the solution of system \eqref{eq:monolithic_alge} we consider the SIMPLE (Semi-Implicit Pressure Linked Equation) method \cite{PATANKAR1980}, that is a splitting algorithm for the solution of Stokes and Navier-Stokes problems. We introduce the additive splitting A = D-H, where D is the diagonal of A and -H is the off-diagonal part. Inverting the system with respect to its diagonal, we obtain the following expression for $\vec{u}$:
\begin{equation*}
\vec{u} = \text{D}^{-1}\left[
    \text{H}\vec{u} - \text{B}^\intercal\vec{p} + \vec{F}
\right] = \tilde{\vec{u}}-\text{D}^{-1}\text{B}^\intercal\vec{p},
\end{equation*}
where $\tilde{\vec{u}}=\text{D}^{-1}\left[\text{H}\vec{u} + \vec{F}\right]$, and we substitute it in the second equation of system \eqref{eq:monolithic_alge} to obtain
\begin{equation*}
\text{B}\tilde{\vec{u}}-\text{B}\text{D}^{-1}\text{B}^\intercal\vec{p}=\vec{0}.
\end{equation*}

The term $\text{B}\text{D}^{-1}\text{B}^\intercal\vec{p}$ is the algebraic counterpart of a Laplacian problem for the pressure where the diffusivity coefficient is $\text{D}^{-1}$, piecewise constant on boxes. This co-located discretization procedure is known to generate spurious pressure modes \cite{FERZIGER2020} because the diagonal dominance of the algebraic system is no more ensured, due to the too large bandwidth of the matrix. To stabilize the problem we employ the so called \textit{Rhie-Chow interpolation} \cite{FERZIGER2020, ZHANG2014, MANGANI2015}, that basically substitutes the term $\text{B}\text{D}^{-1}\text{B}^\intercal\vec{p}$ with the discretization of a pressure Laplacian, using $\text{D}^{-1}$ as a diffusion. Thus, the discrete form of Rhie-Chow stabilization reads
\begin{equation}
\begin{aligned}
\tilde{s}_B(p_B,z_h)
=& \sum_{\bx_i \in \B_h}\sum_{\bx_j\in G_i}\abs{\face_{ij}}
\Biggl[ w_{ij}\text{D}^{-1}_i \sum_{\bx_k\in G_i}
\abs{\face_{ik}}
\left(w_{ik}\PtoBox p_i+ (1-w_{ik})\PtoBox p_k \right)
\vec{n}_{ik}\\
&\qquad+ (1-w_{ij})\text{D}^{-1}_j \sum_{\bx_k\in G_j}
\abs{\face_{jk}}
\left(w_{jk}\PtoBox p_j + (1-w_{jk})\PtoBox p_k\right)\vec{n}_{jk}
\Biggr] \cdot \vec{n}_{ij} z_h\\
&- \sum_{\bx_i \in \B_h}\sum_{\bx_j\in G_i}
\abs{\face_{ij}}\left[w_{ij}\abs{\bx_i}\text{D}_i^{-1} + (1-w_{ij})\abs{\bx_j}\text{D}^{-1}_j\right]
\cfrac{\PtoBox p_j - \PtoBox p_i}{{d_{ij}}}
z_h.\\
\end{aligned}
\label{eq:rhiechow:stab}
\end{equation}

Upon introducing for any $q_h \in \V_h$ the Gauss-Green gradient $\nabla_{h,i}q_h$  of $q_h$ evaluated on box $i$:
$$\abs{\bx_i}\nabla_{h,i} q_h:= \sum_{\bx_j\in G_i} \abs{\face_{ij}}\vec{n}_{ij}
\left(w_{ij} \PtoBox q_i + (1-w_{ij})\PtoBox q_j\right)$$
we can rewrite the Rhie-Chow stabilization operator $\tilde{s}_B: \V_h \times \W_h \rightarrow \mathbb{R}$ as follows
\begin{equation}\label{def:rhie_chow}
\tilde{s}_B(q_h, z_h) = \tilde{s}_B^1(q,z_h) - \tilde{s}_B^2(q,z_h),
\end{equation}
where $\tilde{s}_B^1$ and $\tilde{s}_B^2$ can be written explicitly as
\begin{equation}
\begin{aligned}
\tilde{s}_B^1(q_h,z_h) =& \sum_{\face_{ij} \in \F_h}
\int_{\face_{ij}}
\left(w_{ij}\text{D}_i^{-1}\abs{\bx_i}\nabla_{h,i} q_h
+(1-w_{ij})\text{D}_j^{-1}\abs{\bx_i}\nabla_{h,j} p_h\right)
\cdot\vec{n}_{ij} \jump{z_h}_{ij}\ds, \\
\tilde{s}_B^2(q_h,z_h) =& \sum_{\face_{ij} \in \F_h}
\int_{\face_{ij}}
\left(w_{ij}\abs{\bx_i}\text{D}^{-1}_i
+ (1-w_{ij})\abs{\bx_j}\text{D}^{-1}_j\right)
\cfrac{\PtoBox q_j - \PtoBox q_i}{d_{ij}} \jump{z_h}_{ij}\ds,\\
\end{aligned}
\end{equation}
where $\abs{\bx_i}$ is the measure of box $\bx_i$, $\text{D}_i$ are the diagonal coefficients of matrix A. We also define the Auxiliary Rhie-Chow stabilization operator, acting on continuous functions:
\begin{equation}
\begin{aligned}
s_B: H^1(\domain) \times \W_h \rightarrow \mathbb{R}: \quad
& s_B(q, z_h) = s_B^1(q,z_h) - s_B^2(q,z_h),\\
\end{aligned}
\label{eq:stokes:rhiechow}
\end{equation}
where $s_B^1$ and $s_B^2$ can be written explicitly as
\begin{equation}\label{RC:aux}
\begin{aligned}
s_B^1(q,z_h) =& \sum_{\face_{ij} \in \F_h}
\int_{\face_{ij}}
\left(w_{ij}\text{D}_i^{-1}\int_{\bx_i}\nabla q\dx
+(1-w_{ij})\text{D}_j^{-1}\int_{\bx_j}\nabla q\dx\right)
\cdot\vec{n}_{ij} \jump{z_h}_{ij}\ds, \\
s_B^2(q,z_h) =& \sum_{\face_{ij} \in \F_h}
\int_{\face_{ij}}
\left(w_{ij}\abs{\bx_i}\text{D}^{-1}_i
+ (1-w_{ij})\abs{\bx_j}\text{D}^{-1}_j\right)
\pdd{q}{}{\vec{n}_{ij}} \jump{z_h}_{ij}\ds.\\
\end{aligned}
\end{equation}

\begin{remark} \label{rem:rhie_chow_meaning}
The two forms $s_B^1$ and $s_B^2$ have a precise meaning, in particular, the algebraic representation of $s_B^1$ is $\text{B}\text{D}^{-1}\text{B}^\intercal\vec{p}$ and $s^2$ is the scalar counterpart of bilinear form $\tilde{a}_B$ with $\text{D}^{-1}$ as a diffusion.
\end{remark}
\begin{remark} \label{rem:gauss_green_grad}
In $s_B^1$, the gradient on a box is computed using the Gauss-Green theorem. Following the discretization rules used for equations \eqref{eq:stokes:bilforms:discrete}, then for $q_h \in \V_h$ we have:
\begin{equation}
\begin{aligned}
&\abs{\bx_i}\nabla_{h,i} q_h \simeq \int_{\partial \bx_i} q_h \vec{n}_i\ds= \int_{\bx_i}  \nabla \cdot (q_h \text{I})\dx=
\int_{\bx_i}  \nabla q_h\dx,
\end{aligned}
\label{eq:gauss_green_gradient}
\end{equation}
where I is the identity matrix.
\end{remark}

\subsection{Properties of the Rhie-Chow stabilization} \label{sec:rhiechow:props}

We now discuss some important properties of the Rhie-Chow stabilization that will be employed to prove well-posedness and convergence of problem   \eqref{eq:stokes:box}. In particular we discuss
consistency, continuity, coercivity and \textit{inf-sup} stability.

\paragraph{Consistency}
The consistency of Rhie-Chow stabilization is given by:
\begin{equation}
(s_B - \tilde{s}_B)(p_B,\PtoBox q_h)\lesssim
h^{\frac{3}{2}}
\seminorm{p_B}_{H^1}\seminorm{q_h}_*
\label{eq:rhie_chow_consistency}
\end{equation}
where $p_B,q_h\in\V_h$.

To prove it, let us observe that, in view of the definitions of $s_B$  and $\tilde{s}_B$, we have
\begin{equation*}
\begin{aligned}
(s_B - \tilde{s}_B)(p_B,\PtoBox q_h)=&
(s^1_B - \tilde{s}^1_B)(p_B,\PtoBox q_h) -
(s^2_B - \tilde{s}^2_B)(p_B,\PtoBox q_h) \\
=& (I) + (II).
\end{aligned}
\end{equation*}

The second term $(II)=0$ by the fact that
\[\cfrac{\PtoBox p_j - \PtoBox p_i}{d_{ij}}=\pdd{p_B}{}{\vec{n}_{ij}}.\]
On the other hand, we can write $(I)$ as
\begin{equation}
\begin{aligned}
(I)=& \sum_{\face_{ij} \in \F_h}
\int_{\face_{ij}}
\left(w_{ij}\text{D}_i^{-1}\int_{\bx_i}(\nabla p_B - \nabla_{h,i} p_B) \dx\right.\\
& \qquad\qquad\left. +(1-w_{ij})\text{D}_j^{-1}\int_{\bx_j}(\nabla p_B - \nabla_{h,j} p_B)\dx\right)
\cdot\vec{n}_{ij} \jump{q_h}_{ij}\ds.
\end{aligned}
\end{equation}

Without loss of generality, we treat only the case of the first addendum, the second term being similar. Let us first desume that, under Assumption \ref{ass:mesh_regularity}, $\text{D}^{-1}_i$ scales as $h_\triangle^{2-d}$, indeed we have
\begin{equation}
\begin{aligned}
&\text{D}_i =
\sum_{\bx_j\in G_i}\cfrac{\abs{\face_{ij}}}{d_{ij}}
\simeq \sum_{\bx_j\in G_i}h_\triangle^{d-2}, \\
& \cfrac{h_\triangle^{2-d}}{N_{\bx_i}} \lesssim
\text{D}^{-1}_i \lesssim \cfrac{h_\triangle^{2-d}}{N_{\bx_i}}.
\end{aligned}
\label{eq:box:D_estimate}
\end{equation}
 where $N_{B_i}$ denotes the number of faces of box $\bx_i$ (c.f. Section \ref{sec:boxmethod}).

Considering the face barycentres $\faceCentre_{ij}$, since $p_B$ is piecewise linear, the following identity holds:
\begin{equation}
w_{ij}\PtoBox p_i + (1-w_{ij})\PtoBox p_j = p_B(\faceCentre_{ij}),
\label{eq:pl_equality_on_face_centre_0}
\end{equation}
where we employed notation \eqref{eq:notation_ptobox}. Let now $\vec{x}$ be a point on the face $\face_{ij}$. As $p_B$ is piecewise linear, using a Taylor expansion around $\faceCentre_{ij}$ we have:
\begin{equation}
\begin{aligned}
p_B(\vec{x}) =& p_B(\faceCentre_{ij}) + \sum_{\triangle \in \Tau_h: \triangle\cap\dmond_{ij}\neq\emptyset}
(\vec{x} - \faceCentre_{ij})\cdot\nabla p_B
\indicator{\triangle\cap\dmond_{ij}}, \\
\end{aligned}
\label{eq:pl_taylor_exp_on_dmond_0}
\end{equation}
where $\indicator{}$ is the indicator function and $\nabla p_B$ is piecewise constant on each intersection between triangle $\triangle$ and diamond $\dmond_{ij}$ (c.f. Figure \ref{fig:scheme_dual_geom}).

Using equations \eqref{eq:pl_equality_on_face_centre_0} and \eqref{eq:pl_taylor_exp_on_dmond_0}, the Cauchy-Schwarz inequality, Assumption \ref{ass:mesh_regularity}, inequality \eqref{eq:box:D_estimate}, the Inverse Trace inequality (c.f. Lemma \ref{lemma:inverse_trace}) and the H\"older inequality, we get
\begin{equation*}
\begin{aligned}
\sum_{\face_{ij} \in \F_h} \int_{\face_{ij}}&
w_{ij}\text{D}^{-1}_i
\left(\int_{\bx_i}\nabla p_B - \abs{\bx_i}\nabla_{h,i} p_B)\right)
\cdot\vec{n}_{ij} \jump{\PtoBox q_h}_{ij}\ds \\
=& \sum_{\face_{ij} \in \F_h} \int_{\face_{ij}}
w_{ij}\text{D}^{-1}_i \left(
\sum_{\bx_j\in G_i} \int_{\face_{ij}}\vec{n}_{ij}
\left(p_B - w_{ij} \PtoBox p_i - (1-w_{ij})\PtoBox p_j\right)\ds\right)
\cdot\vec{n}_{ij} \jump{\PtoBox q_h}_{ij}\ds \\
=& \sum_{\face_{ij} \in \F_h} \int_{\face_{ij}}
w_{ij}\text{D}^{-1}_i \left(
\sum_{\bx_j\in G_i} \int_{\face_{ij}}\vec{n}_{ij}
(\vec{x}-\faceCentre_{ij})\cdot\nabla p_B\ds
\right)
\cdot\vec{n}_{ij} \jump{\PtoBox q_h}_{ij}\ds \\
\lesssim& \sum_{\face_{ij} \in \F_h}
\left(
h_\triangle^{d-1} \cfrac{h_\triangle^{4 - 2d}}{N_{\bx_i}} \left(
\sum_{\bx_j\in G_i} \int_{\face_{ij}}\vec{n}_{ij}
(\vec{x}-\faceCentre_{ij})\cdot\nabla p_B\ds
\right)^2 \right)^{\frac{1}{2}}
\left(\int_{\face_{ij}}\jump{\PtoBox q_h}_{ij}^2\ds\right)^{\frac{1}{2}} \\
\leq& \sum_{\face_{ij} \in \F_h}
\left( \cfrac{h_\triangle^{3-d}}{N_{\bx_i}}
\left(\sum_{\bx_j\in G_i} \int_{\face_{ij}} \abs{\vec{n}_{ij}}^2\ds\right)
\left(\sum_{\bx_j\in G_i}\int_{\face_{ij}}
\abs{\vec{x}-\faceCentre_{ij}}^2\abs{\nabla p_B}^2\ds
\right)
\right)^{\frac{1}{2}}
\left(\int_{\face_{ij}}\jump{\PtoBox q_h}_{ij}^2\ds\right)^{\frac{1}{2}} \\
\leq& \sum_{\face_{ij} \in \F_h}
\left( N_{\bx_i}\cfrac{h_\triangle^{3-d}}{N_{\bx_i}}h_\triangle^{d-1}h_\triangle^2
\left(\sum_{\bx_j\in G_i}\norm{\nabla p_B}^2_{L^2(\face_{ij})}
\right)
\right)^{\frac{1}{2}}
\left(\int_{\face_{ij}}\jump{\PtoBox q_h}_{ij}^2\ds\right)^{\frac{1}{2}} \\
\leq& \sum_{\face_{ij} \in \F_h}
\left( C^2_{inv}N_{\bx_i}h_\triangle^{4}h_\triangle^{-2}
\norm{\nabla p_B}^2_{L^2(\bx_i)}
\right)^{\frac{1}{2}}
h_\triangle^{\frac{1}{2}}\left(d_{ij}\int_{\face_{ij}}
\abs{\pdd{q_h}{}{\vec{n}_{ij}}}^2\ds\right)^{\frac{1}{2}} \\
\leq& \sum_{\face_{ij} \in \F_h} C_{inv}
\sqrt{N_{\bx_i}}h_\triangle^{\frac{3}{2}}
\left( \norm{\nabla p_B}^2_{L^2(\bx_i)} \right)^{\frac{1}{2}}
\left(d_{ij}\int_{\face_{ij}}
\abs{\pdd{q_h}{}{\vec{n}_{ij}}}^2\ds\right)^{\frac{1}{2}} \\
\leq& C_{inv}\sqrt{N_\bx} h^{\frac{3}{2}}
\left( \sum_{\face_{ij} \in \F_h}
\norm{\nabla p_B}^2_{L^2(\bx_i)} \right)^{\frac{1}{2}}
\left(\sum_{\face_{ij} \in \F_h}d_{ij}\int_{\face_{ij}}
\abs{\pdd{q_h}{}{\vec{n}_{ij}}}^2\ds\right)^{\frac{1}{2}} \\
\lesssim& C_{inv}\sqrt{N_\bx} h^{\frac{3}{2}}
\seminorm{p_B}_{H^1}\seminorm{q_h}_*.
\end{aligned}
\end{equation*}

\paragraph{Continuity}
In the following, we discuss the continuity of the Auxiliary Rhie-Chow operator \eqref{eq:stokes:rhiechow}
and the continuity of the Rhie-Chow operator \eqref{def:rhie_chow}.

For what concerns the continuity of the Auxiliary Rhie-Chow operator, we distinguish two cases, depending on the first entry of the bilinear form $s_B(\cdot,\cdot)$: (a) infinite dimensional, (b) finite dimensional.
Let us first consider the infinite dimensional case, namely we aim at proving the following
\begin{equation}
s_B(p,\PtoBox q_h)\lesssim h^{\frac{3}{2}}\seminorm{\nabla p}_{h,1}\seminorm{q_h}_{\triangle,*}
\label{eq:rhie_chow_continuity_cont}
\end{equation}
where $p \in \Q \cap H^2_{\text{loc}}(\domain)$ and $q_h\in\V_h$.

To prove it, using equations \eqref{eq:stokes:rhiechow} and \eqref{RC:aux}, we have
\begin{equation*}
\begin{aligned}
s_B(p,\PtoBox q_h)= & \sum_{\face_{ij} \in \F_h} \int_{\face_{ij}}
\left(w_{ij}\text{D}^{-1}_i \int_{\bx_i} \nabla p\dx
+ (1-w_{ij})\text{D}^{-1}_j\int_{\bx_j}\nabla p\dx\right)
\cdot\vec{n}_{ij} \jump{\PtoBox q_h}_{ij}\ds \\
&- \sum_{\face_{ij} \in \F_h} \int_{\face_{ij}}
\left(w_{ij}\abs{\bx_i}\text{D}^{-1}_i
+ (1-w_{ij})\abs{\bx_j}\text{D}^{-1}_j\right)
\pdd{p}{}{\vec{n}_{ij}}
\jump{\PtoBox q_h}_{ij}\ds \\
=& \sum_{\face_{ij} \in \F_h} \int_{\face_{ij}}
w_{ij}\text{D}^{-1}_i
\left(\int_{\bx_i} \nabla p\dx - \abs{\bx_i}\nabla p\right)
\cdot\vec{n}_{ij} \jump{\PtoBox q_h}_{ij}\ds \\
&+\sum_{\face_{ij} \in \F_h} \int_{\face_{ij}}(1-w_{ij})\text{D}^{-1}_j
\left(\int_{\bx_j}\nabla p\dx - \abs{\bx_j}\nabla p\right)
\cdot\vec{n}_{ij} \jump{\PtoBox q_h}_{ij}\ds. \\
\end{aligned}
\end{equation*}

Without loss of generality, we treat only the case of the first addendum, the second term being similar. Using the Cauchy-Schwarz inequality, the Inverse trace inequality (c.f. Lemma \ref{lemma:inverse_trace}), Assumption \ref{ass:mesh_regularity}, the Poincar\`e inequality (c.f. Lemma \ref{lemma:poincare_on_smooth}), the H\"older inequality and standard interpolation error estimates, we get
\begin{equation*}
\begin{aligned}
\sum_{\face_{ij} \in \F_h} \int_{\face_{ij}}&
\cfrac{1}{2}\text{D}^{-1}_i
\left(\int_{\bx_i} \nabla p\dx - \abs{\bx_i}\nabla p\right)
\cdot\vec{n}_{ij} \jump{\PtoBox q_h}_{ij}\ds \\
\leq& \sum_{\face_{ij} \in \F_h} \left(\int_{\face_{ij}}
\cfrac{1}{4}\text{D}^{-2}_i
\left(\int_{\bx_i} \nabla p\dx
- \abs{\bx_i}\nabla p\right)^2\ds\right)^{\frac{1}{2}} \left(\int_{\face_{ij}}d_{ij}^2
\cfrac{\jump{\PtoBox q_h}_{ij}^2}{d_{ij}^2}\ds\right)^{\frac{1}{2}} \\
\leq& \sum_{\face_{ij} \in \F_h}
\text{D}^{-1}_i \abs{\bx_i}
\norm{\cfrac{1}{\abs{\bx_i}}\int_{\bx_i} \nabla p\dx
- \nabla p}_{L^2(\face_{ij})}
\sqrt{d_{ij}}\left(d_{ij} \int_{\face_{ij}}
\cfrac{\jump{\PtoBox q_h}_{ij}^2}{d_{ij}^2}\ds\right)^{\frac{1}{2}} \\
\leq& \begin{aligned}[t]
\sum_{\face_{ij} \in \F_h}&
h_\triangle^\frac{1}{2}\text{D}^{-1}_i \abs{\bx_i}\\
&C_{\text{inv}}\left[h_\triangle^{-1}
\norm{\cfrac{1}{\abs{\bx_i}}\int_{\bx_i} \nabla p\dx
- \nabla p}_{L^2(\bx_i)}
+ h_\triangle \seminorm{\cfrac{1}{\abs{\bx_i}}\int_{\bx_i} \nabla p\dx
- \nabla p}_{H^1(\bx_i)}
\right]\\
&\left( d_{ij} \int_{\face_{ij}}
\cfrac{\jump{\PtoBox q_h}_{ij}^2}{d_{ij}^2}\ds\right)^{\frac{1}{2}}
\end{aligned} \\
\leq& \sum_{\face_{ij} \in \F_h}C_{\text{inv}}
h_\triangle^\frac{1}{2}\text{D}^{-1}_i \abs{\bx_i}
\left[C_dh_\triangle^{-1}h_\triangle\seminorm{\nabla p}_{H^1(\bx_i)}
+ h_\triangle \seminorm{\nabla p}_{H^1(\bx_i)}
\right]
\left( d_{ij} \int_{\face_{ij}}
\cfrac{\jump{\PtoBox q_h}_{ij}^2}{d_{ij}^2}\ds\right)^{\frac{1}{2}} \\
%%%%%%%% PAGE BREAK
%\end{aligned}
%\end{equation*}
%\begin{equation*}
%\begin{aligned}
%%%%%%%% PAGE BREAK
\leq& \sum_{\face_{ij} \in \F_h}C_{\text{inv}}
C_1(C_d + h_\triangle)h_\triangle^\frac{1}{2}h_\triangle^{2-d} h_\triangle^d
\seminorm{\nabla p}_{H^1(\bx_i)}
\left( d_{ij} \int_{\face_{ij}}
\cfrac{\jump{\PtoBox q_h}_{ij}^2}{d_{ij}^2}\ds\right)^{\frac{1}{2}} \\
\leq&C_{\text{inv}}
C_1C_dh^\frac{3}{2}
\left(\sum_{\face_{ij} \in \F_h}
\seminorm{\nabla p}_{H^1(\bx_i)}^2\right)^{\frac{1}{2}}
\left(\sum_{\face_{ij}\in\F_h}h_\triangle^2
d_{ij} \int_{\face_{ij}}
\cfrac{\jump{\PtoBox q_h}_{ij}^2}{d_{ij}^2}\ds\right)^{\frac{1}{2}} \\
\leq&{C_{\text{inv}}
C_1C_dh^\frac{3}{2}
\left(\sum_{\face_{ij}\in\F_h}\sum_{\triangle\in\Tau_h:\triangle\cap \bx_i \neq\emptyset}
\seminorm{\nabla p}_{H^1(\triangle)}^2\right)^{\frac{1}{2}}
\seminorm{q_h}_{\triangle,*}} \\
\leq&{C_{\text{inv}}
C_1C_dh^\frac{3}{2}
\left(N_{\bx_i}\sum_{\bx_i\in \B_h}\sum_{\triangle\in\Tau_h:\triangle\cap \bx_i \neq\emptyset}
\seminorm{\nabla p}_{H^1(\triangle)}^2\right)^{\frac{1}{2}}
\seminorm{q_h}_{\triangle,*}} \\
\leq&{C_{\text{inv}}
C_1C_d\sqrt{N_{\bx}}h^\frac{3}{2}
\left(\sum_{\triangle\in\Tau_h}
\seminorm{\nabla p}_{H^1(\triangle)}^2\right)^{\frac{1}{2}}
\seminorm{q_h}_{\triangle,*}} \\
\lesssim& {C_{\text{inv}}C_1C_d\sqrt{N_{\bx}}h^{\frac{3}{2}}
\seminorm{\nabla p}_{h,1}\seminorm{q_h}_{\triangle,*}} \\
\label{eq:rhie_chow_consistency_1}
\end{aligned}
\end{equation*}
where we recall that $\seminorm{\cdot}_{h,1}$ denotes the broken $H^1$-norm on the elements of the primal mesh and where we employed the definition \eqref{eq:starnorm_mesh_dep}.

On the other hand, for the finite dimensional case, we prove the following continuity property for the Auxiliary Rhie-Chow operator:
\begin{equation}
s_B(p_h,\PtoBox q_h)\lesssim h^\frac{1}{2}\seminorm{p_h}_{H^1}\seminorm{q_h}_{\triangle,*}
\label{eq:rhie_chow_continuous}
\end{equation}
for any $p_h\in \mathcal{V}_h$.
From the above computations, we have:
\begin{equation}
\begin{aligned}
s_B(p_h,\PtoBox q_h)=& \sum_{\face_{ij} \in \F_h} \int_{\face_{ij}}
w_{ij}\text{D}^{-1}_i
\left(\int_{\bx_i} \nabla p_h\dx - \abs{\bx_i}\nabla p_h\right)
\cdot\vec{n}_{ij} \jump{\PtoBox q_h}_{ij}\ds \\
&+\sum_{\face_{ij} \in \F_h} \int_{\face_{ij}}
(1-w_{ij})\text{D}^{-1}_j
\left(\int_{\bx_j}\nabla p_h\dx - \abs{\bx_j}\nabla p_h\right)
\cdot\vec{n}_{ij} \jump{\PtoBox q_h}_{ij}\ds.\\
\end{aligned}
\label{eq:rhie_chow_exploited}
\end{equation}
Also in this case, without loss of generality, we treat only the case of the first addendum, the second term being similar.

By the Cauchy-Schwarz inequality, Property \ref{property:constant_on_diamond}, Assumption \ref{ass:mesh_regularity} and the H\"older inequality, we obtain
\begin{equation}
\begin{aligned}
\sum_{\face_{ij} \in \F_h} \int_{\face_{ij}}&
w_{ij}\text{D}^{-1}_i
\left(\int_{\bx_i} \nabla p_h\dx - \abs{\bx_i}\nabla p_h\right)
\cdot\vec{n}_{ij} \jump{\PtoBox q_h}_{ij}\ds \\
=& \sum_{\face_{ij} \in \F_h} \int_{\face_{ij}} w_{ij}\text{D}^{-1}_i
\left(\int_{\bx_i} \nabla p_h \cdot\vec{n}_{ij}\dx -
\cfrac{\abs{\bx_i}}{\abs{\dmond_{ij}}}\int_{\dmond_{ij}}\nabla p_h\cdot\vec{n}_{ij}\dx\right)
 \jump{\PtoBox q_h}_{ij}\ds \\
%%%%%%%% PAGE BREAK
%\end{aligned}
%\end{equation}
%\begin{equation}
%\begin{aligned}
%%%%%%%% PAGE BREAK
\leq& \sum_{\face_{ij} \in \F_h} \int_{\face_{ij}} w_{ij}\text{D}^{-1}_i
\int_{\bx_i} \abs{\nabla p_h \cdot\vec{n}_{ij}}\dx\jump{\PtoBox q_h}_{ij}\ds \\
&+ \sum_{\face_{ij} \in \F_h} \int_{\face_{ij}} w_{ij}\text{D}^{-1}_i
\cfrac{\abs{\bx_i}}{\abs{\dmond_{ij}}}\int_{\dmond_{ij}}
\abs{\nabla p_h\cdot\vec{n}_{ij}}\dx
 \jump{\PtoBox q_h}_{ij}\ds\\
\lesssim& \sum_{\face_{ij} \in \F_h} h_\triangle^{2-d}
\left(\int_{\face_{ij}}
\abs{\int_{\bx_i} \nabla p_h\cdot\vec{n}_{ij}\dx}^2\ds
\right)^{\frac{1}{2}}
\left(d_{ij}^2\int_{\face_{ij}}
\abs{\pdd{q_h}{}{\vec{n}_{ij}}}^2\ds\right)^{\frac{1}{2}}\\
&+ \sum_{\face_{ij} \in \F_h} h_\triangle^{2-d}
\left(\int_{\face_{ij}}
\cfrac{\abs{\bx_i}^2}{\abs{\dmond_{ij}}^2}
\abs{\int_{\dmond_{ij}}
\nabla p_h\cdot\vec{n}_{ij}\dx}^2\ds
\right)^{\frac{1}{2}}
\left(d_{ij}^2\int_{\face_{ij}}
\abs{\pdd{q_h}{}{\vec{n}_{ij}}}^2\ds\right)^{\frac{1}{2}}\\
\lesssim& \sum_{\face_{ij} \in \F_h} h_\triangle^{2-d}
\left(\int_{\face_{ij}}
\int_{\bx_i}\abs{\nabla p_h}^2\dx
\int_{\bx_i}\abs{\vec{n}_{ij}}^2\dx\ds
\right)^{\frac{1}{2}}
\left(d_{ij}^2\int_{\face_{ij}}
\abs{\pdd{q_h}{}{\vec{n}_{ij}}}^2\ds\right)^{\frac{1}{2}}\\
&+ \sum_{\face_{ij} \in \F_h} h_\triangle^{2-d}
\left(\int_{\face_{ij}}\int_{\dmond_{ij}}
\abs{\nabla p_h}^2\dx\int_{\dmond_{ij}}\abs{\vec{n}_{ij}}^2\dx\ds
\right)^{\frac{1}{2}}
\left(d_{ij}^2\int_{\face_{ij}}
\abs{\pdd{q_h}{}{\vec{n}_{ij}}}^2\ds\right)^{\frac{1}{2}}\\
\lesssim& \sum_{\face_{ij} \in \F_h} h_\triangle^{2-d}
\left(h_\triangle^{d-1}h_\triangle^d
\seminorm{p_h}_{H^1(\bx_i)}^2
\right)^{\frac{1}{2}}
\left(h_\triangle d_{ij}\int_{\face_{ij}}
\abs{\pdd{q_h}{}{\vec{n}_{ij}}}^2\ds\right)^{\frac{1}{2}}\\
&+ \sum_{\face_{ij} \in \F_h} h_\triangle^{2-d}
\left(h_\triangle^{d-1}h_\triangle^d\seminorm{p_h}_{H^1(\dmond_{ij})}^2
\right)^{\frac{1}{2}}
\left(h_\triangle d_{ij}\int_{\face_{ij}}
\abs{\pdd{q_h}{}{\vec{n}_{ij}}}^2\ds\right)^{\frac{1}{2}}\\
\lesssim&
\sum_{\face_{ij} \in \F_h} h_\triangle^{2}
\seminorm{p_h}_{H^1(\bx_i)}
\left(d_{ij}\int_{\face_{ij}}
\abs{\pdd{q_h}{}{\vec{n}_{ij}}}^2\ds\right)^{\frac{1}{2}}\\
\lesssim&  h^{\frac{1}{2}} \left(\sum_{\face_{ij} \in \F_h}\seminorm{p_h}_{H^1(\bx_i)}^2\right)^{\frac{1}{2}}
\left(\sum_{\face_{ij} \in \F_h} h_\triangle^3
d_{ij}\int_{\face_{ij}}
\abs{\pdd{q_h}{}{\vec{n}_{ij}}}^2\ds\right)^{\frac{1}{2}} \\
=&h^\frac{1}{2}\seminorm{p_h}_{H^1}\seminorm{q_h}_{\triangle,*}.
\end{aligned}
\end{equation}

For what concerns the continuity for the Rhie-Chow operator \eqref{def:rhie_chow} we show that a similar estimate holds, namely
\begin{equation}
\tilde{s}_B(p_h,\PtoBox q_h)\lesssim h^\frac{1}{2}\seminorm{p_h}_{H^1}\seminorm{q_h}_{\triangle,*}.
\label{eq:rhie_chow_continuous_dicrete}
\end{equation}

Indeed, it is sufficient to consider the consistency property \eqref{eq:rhie_chow_consistency} and the continuity of $s_B$, then we have
\begin{equation}
\begin{aligned}
\tilde{s}_B(p_h,\PtoBox q_h)=&
(\tilde{s}_B - s_B)(p_h,\PtoBox q_h) + s_B(p_h,\PtoBox q_h) \\
\lesssim& (h^{\frac{3}{2}} + h^\frac{1}{2})\seminorm{p_h}_{H^1}\seminorm{q_h}_{\triangle,*}.
\end{aligned}
\label{eq:rhie_chow_exploited_discrete}
\end{equation}

\paragraph{Coercivity}
We analyse now the coercivity properties of the Rhie-Chow stabilization. Let $q_h\in\V_h$ and consider equation \eqref{eq:rhie_chow_exploited}. In view of
\begin{equation}
\seminorm{q_h}_{\triangle,*}^2
\simeq h^3 \seminorm{q_h}_*^2,
\end{equation}
we numerically assess the validity of the following
\begin{conjecture} \label{lemma:rhiechow_coercivity}
The Rhie-Chow stabilization operator satisfies the following coercivity estimate:
\begin{equation}
\tilde{s}_B(q_h,\PtoBox q_h) \gtrsim h^3 \seminorm{q_h}_*^2.\\
\end{equation}
\end{conjecture}
To this aim, we compute the minimum generalized eigenvalue of the Rhie-Chow stabilization with respect to the $*$-norm and show that it decreases $h$ goes to zero. We consider the algebraic counterpart of $s_B$ and $\seminorm{\cdot}_{\triangle,*}$ and we compute
\begin{equation}
R_*(s_B)=\min\limits_{\vec{q}\in \mathbb{R}^{N_p}}
\cfrac{\vec{q}^\intercal\text{S}\vec{q}}
{\vec{q}^\intercal\text{Q}\vec{q}},
\label{eq:rhie_chow_rayleigh}
\end{equation}
where Q is the matrix that represents the $*$-norm of $\vec{q}$.

We generated Voronoi dual grids of a squared domain $\domain=[-1,1]^d$ where $d=2,3$, for different values of $h$. The results are reported in Tables \ref{tab:rhiechow_eigvals} and \ref{tab:rhiechow_eigvals_3d}.

We observe that the minimum generalized eigenvalue diminishes with rate 3 thus suggesting the validity of the Conjecture.
\begin{table}[!ht]
\centering
\renewcommand{\arraystretch}{1.25}
\begin{tabular}{c|cccccc}
\hline
\multicolumn{7}{c}{Coercivity constant of Rhie-Chow stabilization, 2D case}\\
\hline
$h$ & 0.025 & 0.013 & 0.0063 & 0.0031 & 0.0016 & 0.00078\\
$R_*(s_B)$ & 9.6e-07 & 1.3e-07 & 1.7e-08 & 2.2e-09 & 2.8e-10 & 3.3e-11\\
& -- & 2.9 & 2.9 & 3 & 3 & 3\\
\hline
\end{tabular}
\caption{Minimum generalized eigenvalue of Rhie-Chow matrix with respect to $*$-norm computed on a uniform polygonal mesh. It is also reported the diminishing rate of minimum eigenvalues, computed as $\log_2(R_*\vert_h/R_*\vert_{\frac{h}{2}})$, representing the coercivity lower bound of $s_B$.}
\label{tab:rhiechow_eigvals}
\end{table}
\begin{table}[!ht]
\centering
\renewcommand{\arraystretch}{1.25}
\begin{tabular}{c|ccccc}
\hline
\multicolumn{6}{c}{Coercivity constant of Rhie-Chow stabilization, 3D case}\\
\hline
$h$ & 0.1 & 0.05 & 0.025 & 0.013 & 0.0063\\
$R_*(s_B)$ & 9.3e-05 & 1.3e-05 & 1.9e-06 & 2.6e-07 & 3.6e-08\\
& -- & 2.8 & 2.8 & 2.8 & 2.9\\
\hline
\end{tabular}
\caption{Minimum generalized eigenvalue of Rhie-Chow matrix with respect to $*$-norm computed on a uniform polyhedral mesh. It is also reported the diminishing rate of minimum eigenvalues, computed as $\log_2(R_*\vert_h/R_*\vert_{\frac{h}{2}})$, representing the coercivity lower bound of $s_B$ (c.f. equation \eqref{eq:rhie_chow_rayleigh}).}
\label{tab:rhiechow_eigvals_3d}
\end{table}

\paragraph{\textit{inf-sup} stability}
In the following we study the validity of the generalized \textit{inf-sup} property for the Rhie-Chow stabilization.

\begin{conjecture}[Generalized \textit{inf-sup} for Rhie-Chow]
\label{lemma:rhiechow_inf_sup}
$\exists \beta_h > 0$, independent of $h$ s.t.
\begin{equation}
\sup_{\vec{v}_h\in \bs{\V}_h}
\cfrac{\tilde{c}_B(\vec{v}_h, \PtoBox q_h)}
{\seminorm{\vec{v}_h}_{*}}
+ \tilde{s}_B(q_h,\PtoBox q_h)^{\frac{1}{2}} \geq \beta_h \norm{\PtoBox q_h}_{L^2},
\qquad \forall q_h\in\V_h
\label{eq:ctilde_inf_sup}
\end{equation}
\end{conjecture}

To numerically explore the validity of \eqref{eq:ctilde_inf_sup}, we employ a numerical assessment. Having in mind system \eqref{eq:monolithic_alge}, let us consider the Rhie-Chow stabilized monolithic algebraic system corresponding to the Box method formulation of Stokes system \eqref{eq:stokes:box}:
\begin{equation}
\left[
\begin{matrix} \text{A} & \text{B}^\intercal \\ \text{B} & \mat{C} \\ \end{matrix}
\right]
\left[\begin{matrix} \vec{u} \\ \vec{p} \end{matrix} \right]
=  \left[\begin{matrix} \vec{F} \\ \vec{0} \end{matrix} \right]
\label{eq:monolithic_alge_rhiechow}
\end{equation}
where C is the matrix associated to Rhie-Chow stabilization. Equation \eqref{eq:ctilde_inf_sup} corresponds to the following algebraic inequality: $\forall \vec{q}\in\Real^{N_p}$,
\begin{equation}
\sup_{\vec{w}\in\Real^{N_u}}
\cfrac{\vec{q}^\intercal \mat{B} \vec{w}}
{\sqrt{\vec{w}^\intercal\mat{H}\vec{w}}
\sqrt{\vec{q}^\intercal\mat{V}\vec{q}}}
+ \cfrac{\sqrt{\vec{q}^\intercal \mat{C} \vec{q}}}
{\sqrt{\vec{q}^\intercal\mat{V}\vec{q}}} \geq \beta_h
\label{eq:infsup_alge}
\end{equation}
where $\mat{V}\in\Real^{N_p \times N_p}$ is the mass matrix, i.e. a diagonal matrix with box volumes on diagonal ($\mat{V}_{ii}=\abs{\bx_i}$), representing the $L^2$-norm of box-wise constant functions, and $\mat{H}\in\Real^{N_u\times N_u}$ is the matrix representing the $*$-norm in $d$ dimensions. Notice also that by construction of A, it holds $\mat{A}=\nu \mat{H}$.

Choose now $\vec{w} = \mat{A}^{-1}\mat{B}^\intercal\vec{q}$, then we have
\begin{equation}
\begin{aligned}
\sup_{\vec{w}\in\Real^{N_u}}
\cfrac{\vec{q}^\intercal \mat{B} \vec{w}}
{\sqrt{\vec{w}^\intercal\mat{H}\vec{w}}
\sqrt{\vec{q}^\intercal\mat{V}\vec{q}}} \geq&
\cfrac{\vec{q}^\intercal \mat{B} \mat{A}^{-1} \mat{B}^\intercal \vec{q}}
{\sqrt{\vec{q}^\intercal \mat{B} \mat{A}^{-1}
\mat{H}\mat{A}^{-1}\mat{B}^\intercal\vec{q}}
\sqrt{\vec{q}^\intercal\mat{V}\vec{q}}} \\
=& \sqrt{\nu} \cfrac{\vec{q}^\intercal \mat{B} \mat{A}^{-1} \mat{B}^\intercal \vec{q}}
{\sqrt{\vec{q}^\intercal \mat{B} \mat{A}^{-1}
\mat{A}\mat{A}^{-1}\mat{B}^\intercal\vec{q}}
\sqrt{\vec{q}^\intercal\mat{V}\vec{q}}}\\
=& \sqrt{\nu} \cfrac{\sqrt{
    \vec{q}^\intercal \mat{B} \mat{A}^{-1} \mat{B} \vec{q}
}}{\sqrt{\vec{q}^\intercal\mat{V}\vec{q}}}.
\end{aligned}
\label{eq:infsup_alge_2}
\end{equation}
Then, we have
\begin{equation}
\begin{aligned}
\sup_{\vec{w}\in\Real^{N_u}}
\cfrac{\vec{q}^\intercal \mat{B} \vec{w}}
{\sqrt{\vec{w}^\intercal\mat{H}\vec{w}}
\sqrt{\vec{q}^\intercal\mat{V}\vec{q}}}
+ \cfrac{\sqrt{\vec{q}^\intercal \mat{C} \vec{q}}}
{\sqrt{\vec{q}^\intercal\mat{V}\vec{q}}} \geq&
\sqrt{\nu} \cfrac{\sqrt{
    \vec{q}^\intercal \mat{B} \mat{A}^{-1} \mat{B} \vec{q}
}}{\sqrt{\vec{q}^\intercal\mat{V}\vec{q}}}
+ \cfrac{\sqrt{\vec{q}^\intercal \mat{C} \vec{q}}}
{\sqrt{\vec{q}^\intercal\mat{V}\vec{q}}} \\
\geq& \sqrt{\nu} \cfrac{\sqrt{
    \vec{q}^\intercal \left(\mat{B} \mat{A}^{-1} \mat{B}
+\mat{C}\right) \vec{q}
}}{\sqrt{\vec{q}^\intercal\mat{V}\vec{q}}}. \\
\end{aligned}
\label{eq:infsup_alge_3}
\end{equation}

Hence,
\begin{equation}
R_*(\mat{S}_{\text{box}}+\mat{C}) \coloneq
\min\limits_{\vec{q}\in \mathbb{R}^{N_p}}
\cfrac{\vec{q}^\intercal(\mat{S}_{\text{box}} + \mat{C})\vec{q}}
{\vec{q}^\intercal\text{V}\vec{q}}=\beta_h^2
\label{eq:schur_rayleigh}
\end{equation}
where the explicit expression of C is
\begin{equation}
\mat{C} = \text{B}\text{D}^{-1}\text{B}^\intercal - \mat{R}(\text{D}^{-1})
\end{equation}
and $\mat{R}(\text{D}^{-1})$ is the scalar Laplacian matrix computed employing $\text{D}^{-1}$ as diffusivity coefficient.

For different values of $h$, we have generated Voronoi dual grids on a squared domain $\domain=[-1,1]^d$ where $d=2,3$, and we have computed $R_*(\mat{S}_{\text{box}}+\mat{C})$, c.f. equation \eqref{eq:schur_rayleigh}. The results are reported in Tables \ref{tab:schur_eigvals} and \ref{tab:schur_eigvals_3d} and show that the minimum generalized eigenvalue does not diminish with $h$, as guessed in Conjecture \ref{lemma:rhiechow_inf_sup}, thus suggesting that the generalized \textit{inf-sup} holds.

\begin{table}[!ht]
\centering
\renewcommand{\arraystretch}{1.25}
\begin{tabular}{c|cccccc}
\hline
\multicolumn{7}{c}{\textit{inf-sup} constant in \eqref{eq:ctilde_inf_sup}, 2D case}\\
\hline
$h$ & 0.025 & 0.013 & 0.0063 & 0.0031 & 0.0016 & 0.00078\\
$R_*(S_{\text{box}} - \mat{C})$ & 0.13 & 0.13 & 0.14 & 0.14 & 0.15 & 0.15\\
\hline
\end{tabular}
\caption{Minimum generalized eigenvalue of $\mat{S}_{\text{box}}+\mat{C}$ on a uniform polygonal mesh.}
\label{tab:schur_eigvals}
\end{table}
\begin{table}[!ht]
\centering
\renewcommand{\arraystretch}{1.25}
\begin{tabular}{c|ccccc}
\hline
\multicolumn{6}{c}{\textit{inf-sup} constant in \eqref{eq:ctilde_inf_sup}, 3D case}\\
\hline
$h$ & 0.1 & 0.05 & 0.025 & 0.013 & 0.0063\\
$R_*(S_{\text{box}} - \mat{C})$ & 0.06 & 0.055 & 0.058 & 0.062 & 0.066\\
\hline
\end{tabular}
\caption{Minimum generalized eigenvalue of $\mat{S}_{\text{box}}+\mat{C}$ computed on a uniform polyhedral mesh.}
\label{tab:schur_eigvals_3d}
\end{table}

\section{Convergence of RCBM:  empirical study} \label{sec:numerical}
In this section, we present some numerical experiments to empirically explore  the convergence properties of the Rhie-Chow stabilized Box Method (RCBM), cf. \eqref{eq:stokes:box} where $\tilde{s}_B$ has been chosen equal to the RC stabilization operator \eqref{def:rhie_chow}. In particular, we estimate the errors in the $H^1$ norm for the velocity and in the $L^2$ norm for the pressure. The method used to solve the Stokes system is the SIMPLE splitting method. Hinging upon the results of the numerical test, we conjecture the validity of the following error estimate
\begin{equation}
\label{thm:convergence_0}
\norm{\vec{u}_B - \vec{u}}_{H^1} + \norm{p_B - p}_{L^2} \lesssim h
\end{equation}
where $(\vec{u}_B,p_B)$
is the solution of
\eqref{eq:stokes:box}, while ($(\vec{u},p)$) is the exact solution of the Stokes problem \eqref{eq:stokes}.
The theoretical validity of the above estimate will be addressed in Section \ref{sec:theory}, where a proof of the convergence will be obtained under
Conjectures \ref{lemma:rhiechow_coercivity} and \ref{lemma:rhiechow_inf_sup}, whose  validity has been empirically addressed in Section \ref{sec:rhiechow:props}.
In the following, we consider two test cases. For both cases, as mentioned in Section \ref{sec:boxmethod}, the computations have been performed employing a Voronoi dual mesh of a Delaunay triangulation (i.e. connecting circumcentres of triangles with straight lines). To generate the dual mesh in 2D we relied on a tool implemented in OpenFOAM called \texttt{polyDualMesh} modified in order to use circumcentres of triangles and nodes of dual mesh instead of barycentres. For the 3D mesh we relied on a custom implementation of a Voronoi grid generator \texttt{voroToFoam} (\texttt{https://github.com/alfiogn/voroToFoam}), based on OpenFOAM and on the open-source software Voro\texttt{++}. The 2D and 3D meshes are represented in Figure \ref{fig:poly_grids}.
\begin{figure}
\includegraphics[width=0.5\linewidth]{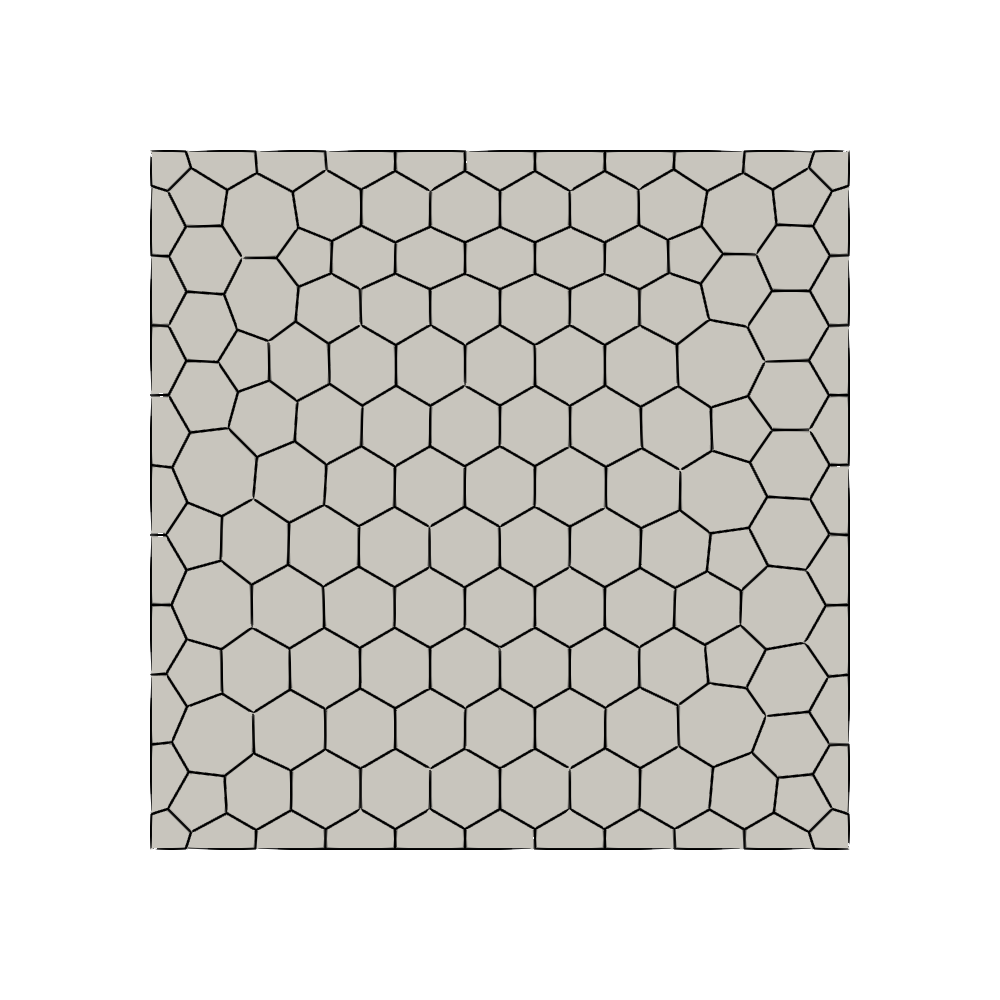}
\includegraphics[width=0.48\linewidth]{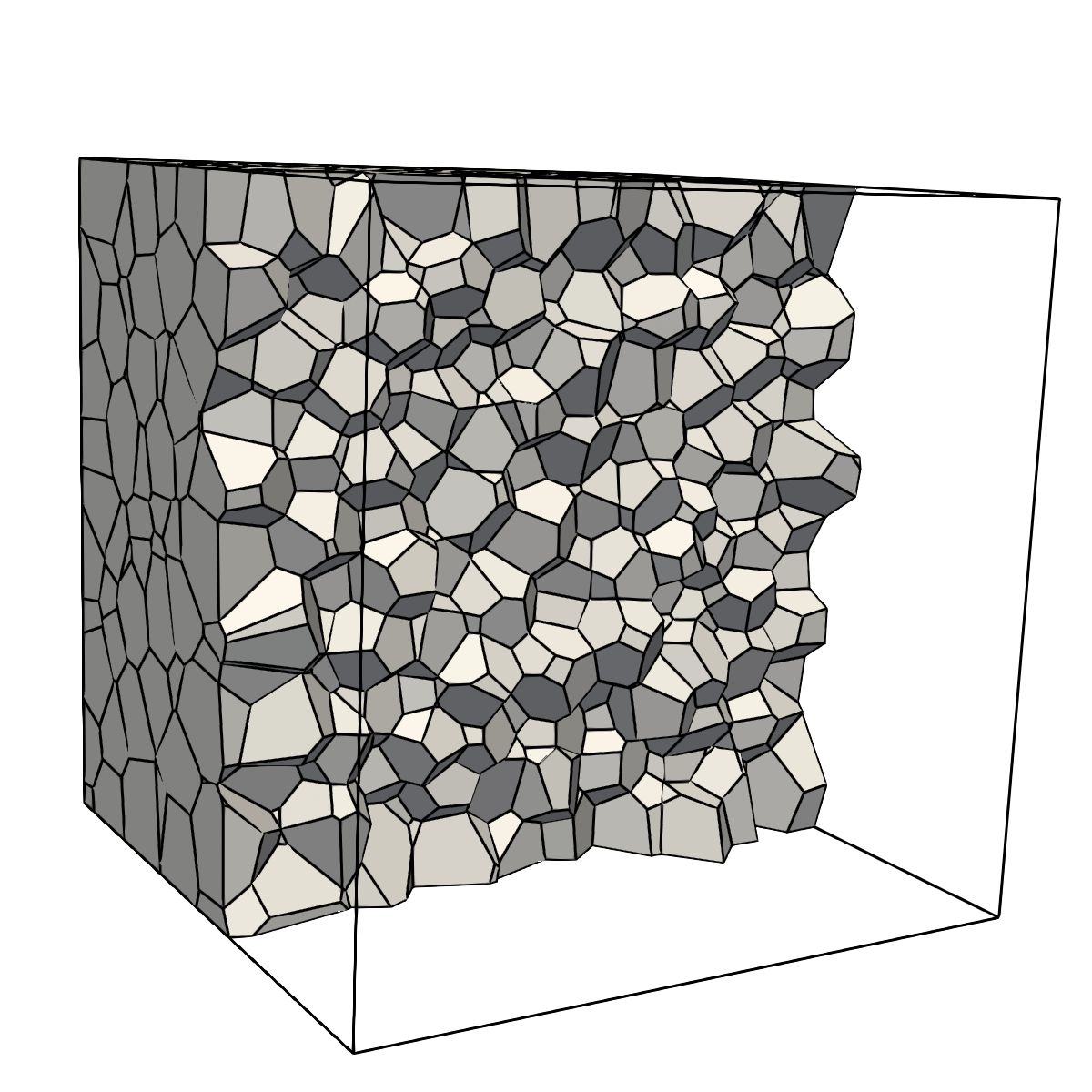}
\caption{Representation of the 2D (left) and the, clipped, 3D (right) Voronoi grids employed in the numerical assessment of convergence properties of the BM.}
\label{fig:poly_grids}
\end{figure}

The first is a 2D case. We consider the domain $\domain=[-1/4,1/4]^2$. We set the analytic solution to
\begin{equation}
\begin{aligned}
\vec{u}=&
\left[\begin{matrix}
- \sin{\left(2 \pi y \right)} \cos{\left(2 \pi x \right)}\\
\sin{\left(2 \pi x \right)} \cos{\left(2 \pi y \right)}
\end{matrix}\right], \\
p = &- \cfrac{1}{4} \left(\cos{\left(4 \pi x \right)}
+ \cos{\left(4 \pi y \right)} \right)
\end{aligned}
\end{equation}
and we consider the following values of $h$:
\[h=0.025,0.0125,0.00625,0.003125.\]

The second is a 3D case. We consider the domain $\domain=[-1/4,1/4]^3$. We set the analytic solution to
\begin{equation}
\begin{aligned}
\vec{u}=&
\left[\begin{matrix}
\sin{\left(2 \pi y \right)} \sin{\left(2 \pi z \right)} \cos{\left(2 \pi x \right)}\\
\sin{\left(2 \pi x \right)} \sin{\left(2 \pi z \right)} \cos{\left(2 \pi y \right)}\\
- 2 \sin{\left(2 \pi x \right)} \sin{\left(2 \pi y \right)} \cos{\left(2 \pi z \right)}\end{matrix}\right]
, \\
p = & - \cfrac{1}{8}\left(
\cos{\left(4 \pi x \right)}
+ \cos{\left(4 \pi y \right)}
+ \cos{\left(4 \pi z \right)}
\right)
\end{aligned}
\end{equation}
and we consider the following values of $h$:
\[h=0.05,0.025,0.0125,0.00625.\]
For both cases we set the boundary conditions accordingly to analytic solutions and $\vec{f}=-\Delta\vec{u} + \nabla p$.

Solving the problems, we obtain the convergence rates of the $H^1$ error for the velocity and $L^2$ error for the pressure, represented in Figure \ref{fig:stokes:conf}. For both velocity $H^1$ error and pressure $L^2$ error we observe a rate of convergence in accordance with Theorem \ref{thm:convergence_0}. In the 3D case, we impute the oscillating behaviour of the pressure error of the 3D case to the fact that is not easy to build a perfectly regular Voronoi grid and this reflects on the accuracy of the method.

\begin{figure}[!ht]
\includegraphics[width=0.49\linewidth]{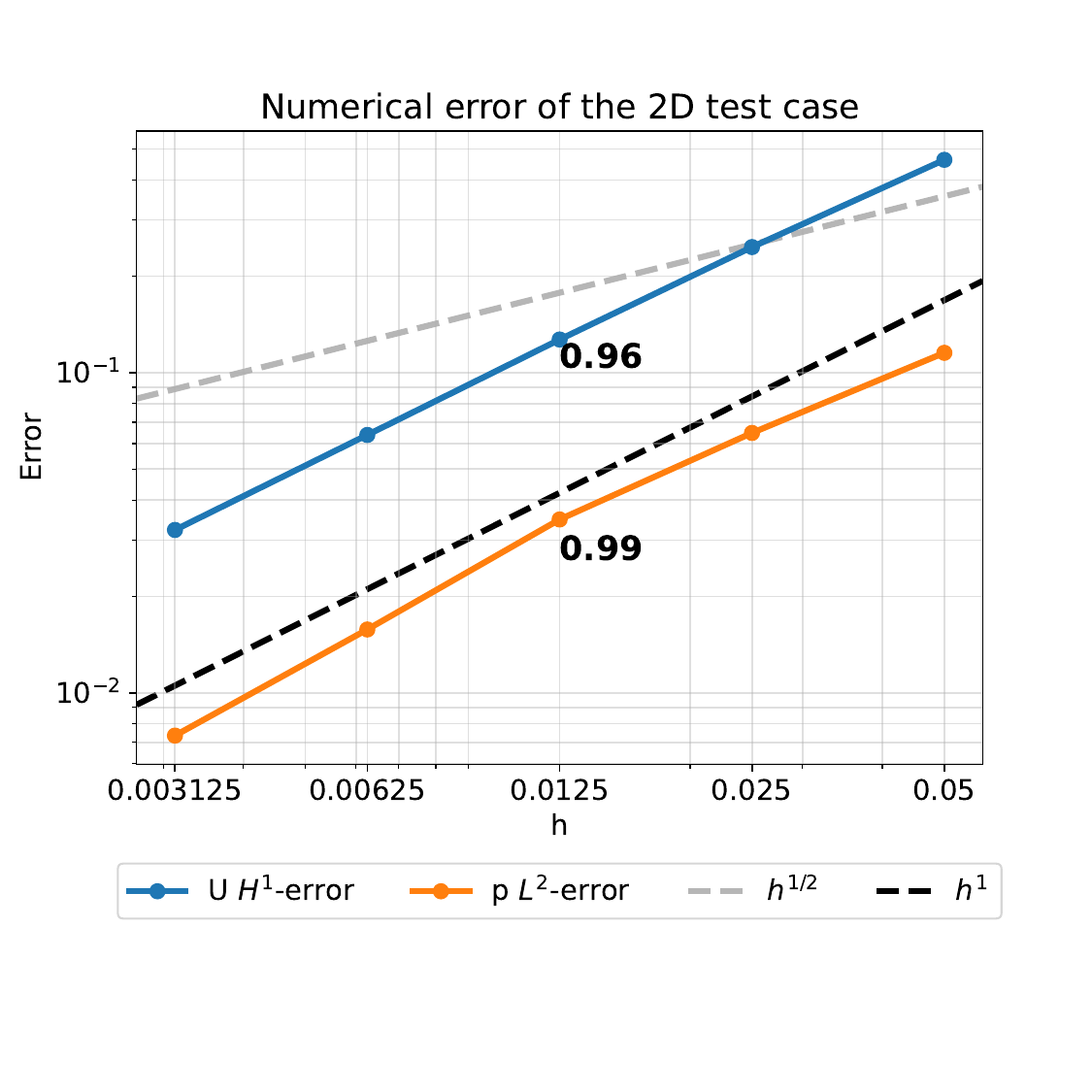}
\includegraphics[width=0.49\linewidth]{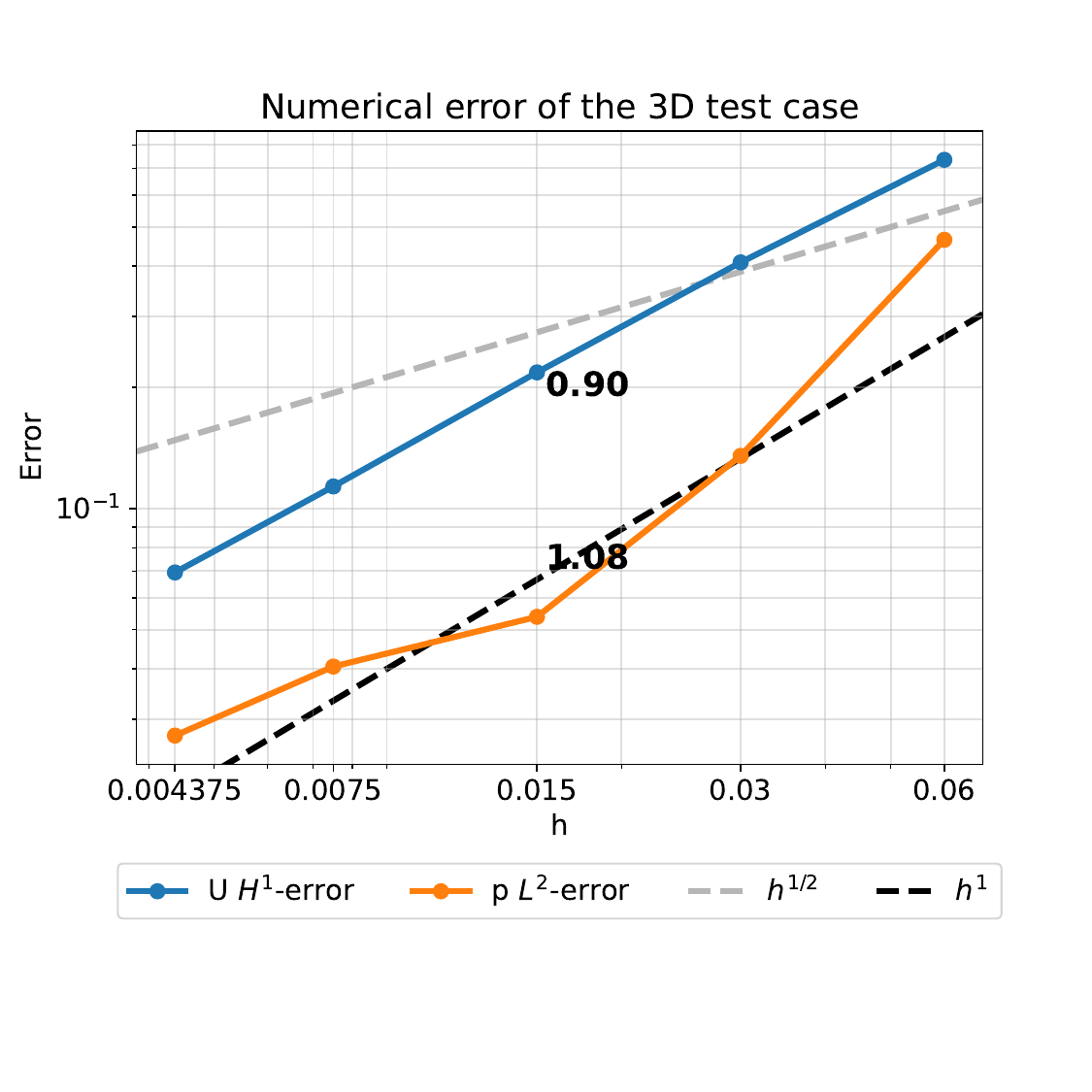}
\caption{Convergence rates of the numerical error of BM solutions. On the left the 2D case errors and on the right the 3D case ones. Numbers are the rates computed using a Least Squares approximation on the log-log plot values.}
\label{fig:stokes:conf}
\end{figure}

\section{Well-posedeness and convergence of RCBM: analysis} \label{sec:theory}
In this section, we address the well-posedeness of RCBM (see Section \ref{sec:wellposed}), together with the theoretical validity of the error estimate \eqref{thm:convergence_0} (cf. Theorem \ref{thm_convergence} below) that will be obtained  under Conjectures \ref{lemma:rhiechow_coercivity} and \ref{lemma:rhiechow_inf_sup}, whose  validity has been empirically addressed in Section \ref{sec:rhiechow:props}.

\subsection{Well-posedness} \label{sec:wellposed}
To prove the well-posedness of problem \eqref{eq:stokes:box}, we first need the following results.
\begin{lemma}[$\widetilde{\C}_B$ consistency] \label{lemma:consistency}
Let $(\vec{u},p) \in \bs{\V} \times (\Q \cap H^2_{\text{loc}}(\domain))$ be the solution to problem \eqref{eq:stokes:weak} and $(\vec{u}_B,p_B)\in \bs{\V}_h \times \V_h$ be the solution to problem \eqref{eq:stokes:box}. Then there holds:
\begin{equation*}
\begin{aligned}
\C_B((\vec{u},p),(\PtoBox \vec{v}_h,\PtoBox q_h))-&
\widetilde{\C}_B((\vec{u}_B,p_B),(\PtoBox \vec{v}_h,\PtoBox q_h))  \\
\lesssim&\left(h^2 \norm{\vec{f}}_{L^2}
+ h\seminorm{p_B}_{H^1}\right)
\trinorm{\vec{v}_h,q_h}_{box}
\end{aligned}
\end{equation*}
$\forall \vec{v}_h \in \bs{\V}_h,\, q_h \in \V_h$ and $\seminorm{\cdot}_{h,1}$ is the $H^1$ broken seminorm.
\end{lemma}
The full proof of Lemma \ref{lemma:consistency} can be found in Appendix \ref{app:proofs}, proof \ref{proof:consistency}.

\begin{lemma}[Continuity of $\C_B$] \label{lemma:boundedness}
Let $(\vec{v},q)\in \bs{\V} \times (\Q \cap H^1(\domain))$ and $(\vec{v}_h,q_h)\in \bs{\V}_h \times \V_h$, then there holds:
\begin{equation*}
\C_B((\vec{v},q),(\PtoBox \vec{v}_h,\PtoBox q_h))\lesssim
\left(\seminorm{\vec{v}}_{H^1}
+ \norm{q}_{L^2} + h^{\frac{3}{2}} \seminorm{\nabla q}_{h,1}\right)\trinorm{\vec{v}_h,q_h}_{box}.
\end{equation*}
\end{lemma}
\begin{lemma}[Continuity of $\widetilde{\C}_B$] \label{lemma:boundedness_tilde}
Let $\vec{v}_h,\vec{w}_h\in \bs{\V}_h$ and $q_h,z_h\in \V_h$, then there holds:
\begin{equation*}
\widetilde{\C}_B((\vec{v}_h,q_h),(\PtoBox \vec{w}_h,\PtoBox z_h))\lesssim
\left(\seminorm{\vec{v}}_{H^1}
+ \norm{q_h}_{L^2} + h^{\frac{1}{2}} \seminorm{q_h}_{H^1}\right)\trinorm{\vec{w}_h,z_h}_{box}.
\end{equation*}
\end{lemma}
The proofs of Lemmas \ref{lemma:boundedness} and \ref{lemma:boundedness_tilde} can be found in Appendix \ref{app:proofs}, proofs \ref{proof:boundedness} and \ref{proof:boundedness_tilde}, respectively.

\begin{lemma}[$\widetilde{\C}_B$ coercivity] \label{lemma:coercivity}
Let $(\vec{v}_h,q_h)\in \bs{\V}_h \times \V_h$, then there holds:
\begin{equation*}
\widetilde{\C}_B((\vec{v}_h,q_h),(\PtoBox \vec{v}_h,\PtoBox q_h))\geq
\nu\seminorm{\vec{v}_h}_{H^1}^2
+ \seminorm{q_h}^2_{\triangle,*}.
\end{equation*}
\end{lemma}
The proof of Lemma \ref{lemma:coercivity} is reported in Appendix \ref{app:proofs} (c.f. proof \ref{proof:coercivity}) where we employ the numerical assessment the Conjecture \ref{lemma:rhiechow_coercivity}.

To conclude the well-posedness analysis, the last element we need is the \textit{inf-sup} stability of the Rhie-Chow stabilized problem whose validity has been conjectured in Conjecture \ref{lemma:rhiechow_inf_sup}.

\begin{theorem}\label{thm:stokes_well_posed}
$\widetilde{\C}_B$ satisfies the discrete \textit{inf-sup} condition:
let $\vec{u}_h,\vec{v}_h\in\bs{\V_h},\,p_h,q_h\in\V_h$, $\exists \gamma>0$, independent of $h$ s.t.
\begin{equation}
\gamma\trinorm{\vec{u}_h,p_h}_{box}\lesssim
\sup_{(\PtoBox \vec{v}_h,\PtoBox q_h)\in\bs{\V}_h\times\V_h}\cfrac{
\widetilde{\C}_B((\vec{u}_h,p_h),(\PtoBox \vec{v}_h,\PtoBox q_h))
}{
\trinorm{\vec{v}_h,q_h}_{box}
}.
\label{eq:ref_sup_0}
\end{equation}
\end{theorem}
\begin{proof}
Employing Lemmas \ref{lemma:coercivity} and \ref{lemma:boundedness_tilde}, we have
\begin{equation}
\begin{aligned}
\seminorm{p_h}_{\triangle,*}^2
+\nu\seminorm{\vec{u}_h}_{H^1}^2 \leq&
\widetilde{\C}_B((\vec{u}_h, p_h),(\PtoBox \vec{v}_h, \PtoBox q_h)) \\
\lesssim& \left(\seminorm{\vec{u}_h}_{H^1} + \norm{p_h}_{L^2} + h^{\frac{1}{2}}\seminorm{p_h}_{H^1}
\right)\trinorm{\vec{v}_h,q_h}_{box}\\
=: & \mathbb{S}\trinorm{\vec{v}_h,q_h}_{box}. \\
\end{aligned}
\label{eq:wp_0}
\end{equation}
Thus $\mathbb{S}$ turns out to be equal to the supremum in equation \eqref{eq:ref_sup_0}.

Now employing the generalized \textit{inf-sup} (c.f. Conjecture \ref{lemma:rhiechow_inf_sup}) and equation \eqref{eq:cbt_bbt_equivalence}, we have
\begin{equation}
\begin{aligned}
\beta_h\norm{\PtoBox p_h}_{L^2} \leq & \sup_{\vec{v}_h\in \bs{\V}_h}
\cfrac{\tilde{c}_B(\vec{v}_h, \PtoBox q_h)}
{\seminorm{\vec{v}_h}_{*}}
+ \tilde{s}_B(p_h,\PtoBox p_h)^{\frac{1}{2}} \\
=& \sup_{\vec{v}_h\in \bs{\V}_h}
\cfrac{-\tilde{b}_B(\PtoBox\vec{v}_h, p_h)}
{\seminorm{\vec{v}_h}_{*}} + \tilde{s}_B(p_h,\PtoBox p_h)^{\frac{1}{2}}\\
=& \sup_{\vec{v}_h\in \bs{\V}_h}
\cfrac{a_B(\vec{u}_h,\PtoBox\vec{v}_h)
- \widetilde{\C}_B((\vec{u}_h,p_h),(\PtoBox\vec{v}_h, 0))
}{\seminorm{\vec{v}_h}_{*}}
+ \tilde{s}_B(p_h,\PtoBox p_h)^{\frac{1}{2}}\\
=& \sup_{\vec{v}_h\in \bs{\V}_h}
\cfrac{a_B(\vec{u}_h,\PtoBox\vec{v}_h)}{\seminorm{\vec{v}_h}_{*}}
+ \tilde{s}_B(p_h,\PtoBox p_h)^{\frac{1}{2}}
- \sup_{\vec{v}_h\in \bs{\V}_h}
\cfrac{\widetilde{\C}_B((\vec{u}_h,p_h),(\PtoBox\vec{v}_h, 0))}
{\trinorm{\vec{v}_h,0}_{box}} \\
=& (I) + (II) + (III).\\
\end{aligned}
\label{eq:wp:inf_sup_0}
\end{equation}
We bound separately each term.

By the continuity of $a_B$ (c.f. proof \ref{proof:boundedness} of Lemma \ref{lemma:boundedness}),
\begin{equation*}
(I) \leq
\cfrac{\nu\seminorm{\vec{u}_h}_{H^1}
\seminorm{\vec{v}_h}_*}{\seminorm{\vec{v}_h}_{*}}\leq
\nu\seminorm{\vec{u}_h}_*.
\end{equation*}
Employing the continuity of $\tilde{s}_B$ (c.f. inequality  \eqref{eq:rhie_chow_continuous_dicrete}), by Proposition \ref{prop:starnorm_properties} and Lemma \ref{lemma:maps_error}, together with equations \eqref{eq:starnorm} and \eqref{eq:starnorm_mesh_dep}, we have
\begin{equation*}
(II)^2 \lesssim h^{\frac{1}{2}}\seminorm{p_h}_{H^1}
\seminorm{p_h}_{\triangle,*}
\lesssim h^2\seminorm{p_h}_*^2 \lesssim h^2h_m^{-2}\norm{\PtoBox p_h}_{L^2}^2
\lesssim  \delta^2 \norm{\PtoBox p_h}_{L^2}^2
\lesssim \delta^2h^2 \seminorm{p_h}_{H^1}^2 + \delta^2\norm{p_h}_{L^2}^2
\end{equation*}
where we recall that $\delta=h/h_m$ (c.f. Assumption \ref{ass:mesh_regularity}).
Using equation \eqref{eq:wp_0},
\begin{equation*}
(III)
\lesssim \mathbb{S}\cfrac{\trinorm{\vec{v}_h,0}_{box}}{\trinorm{\vec{v}_h,0}_{box}}
=\mathbb{S}.
\end{equation*}

Collecting the above estimates and employing the definition of $\mathbb{S}$ together with the fact that we assume $h<1$ (thus giving $h^2\leq h$) we obtain

\begin{equation}\label{eq:wp_1}
\norm{\PtoBox p_h}_{L^2}^2 \lesssim \seminorm{\vec{u}_h}_*^2
+ \delta^2 h^2 \seminorm{p_h}^2_{H^1} + \delta^2 \norm{p_h}^2_{L^2}
+ \mathbb{S}^2 \lesssim \mathbb{S}^2
\end{equation}
which, in combination with definition \eqref{eq_norm}, equations \eqref{eq:wp_0} and \eqref{eq:wp_1} and the Young inequality with a suitable $\varepsilon>0$, yields
\begin{equation}
\begin{aligned}
\trinorm{\vec{u}_h,p_h}_{box}^2\lesssim&\seminorm{p_h}_{\triangle,*}^2 + \norm{\PtoBox p_h}_{L^2}^2
+\nu\seminorm{\vec{u}_h}_{H^1}^2\\
\lesssim& \mathbb{S}\trinorm{\vec{u}_h,p_h}_{box}
+ \mathbb{S}^2 \\
\lesssim & \cfrac{1}{2\varepsilon}\mathbb{S}^2
+ \cfrac{\varepsilon}{2}\trinorm{\vec{u}_h,p_h}_{box}^2 + \mathbb{S}^2.
\end{aligned}
\end{equation}
Bringing $\frac{\varepsilon}{2}\trinorm{\vec{u}_h,p_h}_{box}^2$ to the left-hand-side yields the thesis:
\begin{equation}
\trinorm{\vec{u}_h,p_h}_{box}\lesssim
\mathbb{S}.
\label{eq:bnb_condition}
\end{equation}

\end{proof}
The well-posedness of problem \eqref{eq:stokes:box} follows from Banach-Ne\v{c}as-Babu\v{s}ka (BNB) condition, or discrete \textit{inf-sup} condition \cite{BABUSKA1971, QUARTERONI, DIPIETRO2011}.

\subsection{Convergence analysis} \label{sec:convergence}
The convergence of the RCBM is guaranteed by the following theorem.
\begin{theorem}[Convergence] \label{thm_convergence}
Let $\domain$ be of class $C^2$ and let $\vec{f} \in [H^1(\domain)]^d$.
Let $(\vec{u},p) \in (\bs{\V}\cap[H^2(\domain)]^d)\times(\Q\cap H^1(\domain)\cap H^2_{\text{loc}}(\domain))$ the solution to problem \eqref{eq:stokes:weak}. Then
\begin{equation*}
\norm{\vec{u}_B - \vec{u}}_{H^1} + \norm{p_B - p}_{L^2} \lesssim h.
\end{equation*}
\end{theorem}
\begin{proof}
Let us define $I\vec{u},Ip\in \bs{\V}_h\times\V_h$ be the Lagrangian linear interpolations of the exact solutions $\vec{u},p$, respectively. Define also $\vareps=\vec{u} - I\vec{u},\,\vareps_B=I\vec{u} - \vec{u}_B$ and $\eta=p - Ip,\,\eta_B=Ip - p_B$. Employing Lemma \ref{lemma:coercivity} we have
\begin{equation}
\begin{aligned}
\seminorm{\eta_B}_{\triangle,*}^2
+\nu\seminorm{\vareps_B}_{H^1}^2 \leq&
\widetilde{\C}_B((\vareps_B, \eta_B),(\PtoBox \vareps_B, \PtoBox \eta_B))
\pm b_B(\PtoBox\vareps_B,p) \\
&\pm c_B(\vec{u},\PtoBox\eta_B)\pm s_B(p,\PtoBox\eta_B)\\
=& \C_B((\vareps, \eta), (\PtoBox \vareps_B, \PtoBox \eta_B))
+ (\tilde{b}_B - b_B)(\PtoBox\vareps_B, Ip)\\
&+ (\tilde{c}_B - c_B)(I\vec{u}, \PtoBox\eta_B)
+ (\tilde{s}_B - s_B)(Ip, \PtoBox\eta_B)\\
&+ \C_B((\vec{u}, p),(\PtoBox \vareps_B, \PtoBox \eta_B))
- \widetilde{\C}_B((\vec{u}_B, p_B),(\PtoBox \vareps_B, \PtoBox \eta_B)).\\
\end{aligned}
\label{eq:coerc_estimate_0_first_passage}
\end{equation}

Now, using the continuity of $\C_B$ (Lemma \ref{lemma:boundedness}) on the first term, equation \eqref{eq:interp_discr_consistency} on second and third terms and consistency (Lemma \ref{lemma:consistency}) on the fourth and fifth terms, we get the following:
\begin{equation}
\begin{aligned}
\seminorm{\eta_B}_{\triangle,*}^2
+\nu\seminorm{\vareps_B}_{H^1}^2 \leq& \widetilde{\C}_B((\vareps_B, \eta_B),(\PtoBox \vareps_B, \PtoBox \eta_B))\\
\lesssim&\left(
\seminorm{\vareps}_{H^1} + \norm{\eta}_{L^2} + h\seminorm{\nabla\eta}_{h,1}
\right)\trinorm{\vareps_B,\eta_B}_{box}\\
&+ h\seminorm{Ip}_{H^1}\seminorm{\vareps_B}_{H^1}
+ \left(h^2 \norm{\vec{f}}_{L^2}
+ h\seminorm{p_B}_{H^1}
+ h^{\frac{3}{2}}\seminorm{\nabla p}_{h,1}\right)
\trinorm{\vareps_B,\eta_B}_{box}.
\label{eq:coerc_estimate_0}
\end{aligned}
\end{equation}
Observing that
\[\seminorm{\nabla\eta}_{h,1}^2=
\sum_{\triangle \in \Tau_h}\int_{\triangle} D_h^2( p - Ip)\dx
=\sum_{\triangle \in \Tau_h}\int_{\triangle} D^2p\dx
= \seminorm{\nabla p}_{h,1}^2,\]

where $D_h^2$ is the broken Hessian operator, equation \eqref{eq:coerc_estimate_0} becomes
\begin{equation}
\begin{aligned}
\seminorm{\eta_B}_{\triangle,*}^2
+\nu\seminorm{\vareps_B}_{H^1}^2
\lesssim& \left(\seminorm{\vareps}_{H^1} + \norm{\eta}_{L^2} + h\right)
\trinorm{\vareps_B,\eta_B}_{box}\\
=:& \,\mathbb{S}\,\trinorm{\vareps_B,\eta_B}_{box}
\label{eq:coerc_estimate_1}
\end{aligned}
\end{equation}
where, we employed the stability of the Lagrangian interpolant and the continuity with respect to data of the continuous and box solutions (c.f. Theorems \ref{thm:stokes_well_posed_0} and \ref{thm:stokes_well_posed}).

Now employing equation \eqref{eq:bnb_condition} for $\vareps_B$ and $\eta_B$, we have
\begin{equation}
\begin{aligned}
\cfrac{1}{2}\trinorm{\vareps_B,\eta_B}_{box}^2\lesssim&\,
\mathbb{S}^2=\left(\seminorm{\vareps}_{H^1} + \norm{\eta}_{L^2}
+ h\right)^2\\
=& \left(\seminorm{\vareps}_{H^1} + \norm{\eta}_{L^2}
+ h\right)^2.
\end{aligned}
\label{eq:almost_last_squared}
\end{equation}

To conclude the proof, the triangular inequality and Proposition \ref{prop:starnorm_properties} yield
\begin{equation}
\begin{aligned}
\seminorm{\vec{u}_B - \vec{u}}_{H^1} + \norm{p_B - p}_{L^2} \leq&
\seminorm{\vec{u}_B - I\vec{u}}_{H^1} + \norm{p_B - Ip}_{L^2} +
\seminorm{I\vec{u} - \vec{u}}_{H^1} + \norm{Ip - p}_{L^2} \\
\lesssim& \seminorm{\vareps_B}_{H^1} + \norm{\PtoBox\eta_B}_{L^2} +
\seminorm{I\vec{u} - \vec{u}}_{H^1} + \norm{Ip - p}_{L^2} + h\\
\lesssim& \seminorm{I\vec{u} - \vec{u}}_{H^1} + \norm{Ip - p}_{L^2} + \trinorm{\vareps_B, \eta_B}_{box} + h
\end{aligned}
\end{equation}
where we used Lemma \ref{lemma:maps_error} to get
\begin{equation}
\begin{aligned}
\norm{p_B - Ip}_{L^2} \leq&
\norm{p_B - \PtoBox p_B}_{L^2}
+ \norm{\PtoBox p_B - \PtoBox Ip}_{L^2}
+ \norm{Ip - \PtoBox Ip}_{L^2} \\
\lesssim& h\seminorm{p_B}_{H^1}
+ h\seminorm{Ip}_{H^1}
+ \norm{\PtoBox \eta_B}_{L^2}.
\end{aligned}
\end{equation}
Finally, employing inequality \eqref{eq:almost_last_squared}, we obtain the following estimate:
\begin{equation}
\seminorm{\vec{u}_B - \vec{u}}_{H^1} + \norm{p_B - p}_{L^2} \lesssim
\norm{I\vec{u}-\vec{u}}_{H^1} +
\norm{Ip-p}_{L^2}
+ h,
\label{eq:err:best_interpolation_estimate}
\end{equation}
that, using the interpolation estimates  and neglecting higher order terms in $h$ reads:
\begin{equation}
\begin{aligned}
\seminorm{\vec{u}_B - \vec{u}}_{H^1} + \norm{p_B - p}_{L^2} \lesssim&
h,
\end{aligned}
\label{eq:err:final_estimate}
\end{equation}
which is the desired estimate.
\end{proof}

\begin{remark} \label{rem:regularity_assum}
In Theorem \ref{thm_convergence} we used quite strong regularity assumptions: domain $\domain$ of class $C^2$ and the source term $\vec{f} \in [H^1(\domain)]^d$. This is because we need the pressure field to belong to $H^1(\domain)$ and also to $H^2_{\text{loc}}(\domain)$ \cite[Theorems IV.4.1, IV.6.1]{GALDI2011} to satisfy assumptions of Lemma \ref{lemma:consistency}. In particular, these assumptions are needed when dealing with the Rhie-Chow stabilization, that directly involves pressure gradient, indeed the regularity assumptions are employed when proving consistency, continuity and coercivity of the stabilization term (c.f. Lemmas \ref{lemma:consistency}, \ref{lemma:boundedness_tilde} and \ref{lemma:coercivity}).
\end{remark}

\section{Conclusions}

In this work, we considered the Rhie-Chow stabilized Box Method (RCBM) for the numerical approximation of the Stokes problem. In particular, the Rhie-Chow stabilization, a well-known stabilization technique for FVM, has been employed to stabilize the classical Box Method.
In the first part of the paper we provided a variational formulation of the RC stabilization and discussed the validity of crucial properties relevant for the well-posedeness and convergence of RCBM. Then we numerically explored the convergence properties of the RCBM on 2D and 3D test cases. Finally, in the last part of the paper, we theoretically justified the well posedeness of RCBM and the experimentally observed  convergence rates. To tackle this latter issue we built upon some assumptions, whose validity has been numerically explored.

The results contained in this work can be potentially extended to other relevant differential problems in the context of computational fluid dynamics, like the Navier-Stokes equations \cite{WEN2016} and the non-Newtonian Navier-Stokes equations, where the main difficulty is represented by the nonlinear nature of the problem. Moreover, the study of upwind-based schemes, that are more suitable for advection dominated problems, can be potentially faced in the present framework. Ultimately, the analysis of the Box method contained in this work sets a base workflow to theoretically deal, through the lens of the variational framework, with the Finite Volume methods actually implemented in OpenFOAM, the leading open-source software for CFD in industrial applications.

\section{Acknowledgements}
The research has been partially funded by Italian Ministry of Universities and Research (MUR) grant Dipartimento di Eccellenza 2023-2027.
G.N. acknowledges the financial support of Fondazione Politecnico.
N.P. and M.V. have been partially funded by and PRIN2020 n. \verb+20204LN5N5+\,\emph{``Advanced polyhedral discretisations of heterogeneous PDEs for multiphysics problems''}.
G.N., N.P. and M.V. are members of INdAM-GNCS.

\appendix


\section{Appendix}  \label{app:proofs_a}
Before reporting the proofs of some Lemmas of Section \ref{sec:wellposed}, we recall the following results.

The first is for the lumping map \eqref{eq:lumping_map}:
\begin{lemma} \label{lemma:maps_error}
Let $\triangle \in \Tau_h,\,\forall v_h \in \V_h$,
\begin{equation*}
\begin{aligned}
\int_\triangle v_h - \PtoBox v_h\dx=& 0,\\
\norm{v_h - \PtoBox v_h}_{L^2(\triangle)}
\leq& C h_\triangle \seminorm{v_h}_{H^1(\triangle)}, \\
\norm{\PtoBox v_h}_{L^2(\triangle)} \leq&
h \seminorm{p_h}_{H^1} + \norm{v_h}_{L^2(\triangle)},
\end{aligned}
\end{equation*}
\end{lemma}
\begin{proof}
Given the results \cite[Equations (2.11) and (2.12)]{QUARTERONI2011}, we only have to prove the third inequality. By the first inequality fo the former Lemma, we have
\begin{equation*}
\begin{aligned}
\norm{\PtoBox v_h}_{L^2(\triangle)} \leq&
\norm{\PtoBox v_h - v_h}_{L^2(\triangle)} + \norm{v_h}_{L^2(\triangle)} \\
\leq& \sqrt{
    \sum_{i=1}^3 \int_{\triangle_i}
    \left(\PtoBox v_h\vert_{\triangle_i} - v_h\vert_{\triangle_i}\right)^2 \dx
} + \norm{v_h}_{L^2(\triangle)} \\
\leq& \sqrt{
    \sum_{i=1}^3 h^2\int_{\triangle_i}
    \abs{\nabla v_h}^2 \dx
} + \norm{v_h}_{L^2(\triangle)} \\
=& h \seminorm{p_h}_{H^1} + \norm{v_h}_{L^2(\triangle)}.
\end{aligned}
\end{equation*}
\end{proof}

Notice that, by relationship \eqref{eq:lumping_map}, bilinear forms \eqref{eq:stokes:bilforms} satisfy the following lemma:
\begin{lemma} \label{lemma:bilforms_equivalence}
$\forall  \vec{u}_h,\vec{v}_h \in \bs{\V}_h,\, p_h,q_h \in \V_h,\,$ the following hold
\begin{equation*}
\begin{aligned}
a_B(\vec{u}_h,\PtoBox\vec{v}_h) =& a(\vec{u}_h,\vec{v}_h), \\
b_B(\PtoBox\vec{v}_h,p_h) =& b(\vec{v}_h,p_h), \\
c_B(\vec{u}_h,\PtoBox q_h) =& -b(\vec{u}_h,q_h).
\end{aligned}
\end{equation*}
\end{lemma}
\begin{proof}
By \cite[Lemma 3.2]{QUARTERONI2011}, we just need to prove the last equality. We apply Lemma \ref{lemma:maps_error}:
\begin{equation}
\begin{aligned}
c_B(\vec{u}_h,\PtoBox q_h) =&
 \sum_{\bx \in \B_h} \PtoBox q_h
\int_{\partial \bx} \vec{u}_h \cdot \vec{n}_b\ds \\
=& \sum_{\bx \in \B_h}
\int_{\bx} \vec{u}_h \cdot \nabla\PtoBox q_h\dx
+ \sum_{\bx \in \B_h}
\int_{\bx} \nabla \cdot \vec{u}_h \PtoBox q_h \dx\\
=& \sum_{\triangle \in \Tau_h} \sum_{i=1}^3 \int_{\triangle_i}
\nabla \cdot \vec{u}_h \PtoBox q_h\dx \\
=& \sum_{\triangle \in \Tau_h} \sum_{i=1}^3 \int_{\triangle_i}
\nabla \cdot \vec{u}_h \left(\PtoBox q_h - q_h\right)\dx
+ \sum_{\triangle \in \Tau_h} \sum_{i=1}^3 \int_{\triangle_i}
\nabla \cdot \vec{u}_h q_h\dx \\
=& \sum_{\triangle \in \Tau_h} \int_{\triangle}
\nabla \cdot \vec{u}_h q_h \dx= -b(\vec{u}_h,q_h)
\end{aligned}
\end{equation}
where $\triangle_i,i=1,2,3$, are the partitions of the triangle $\triangle$ formed by the faces of intersecting boxes (see Figure \ref{fig:tri_and_dual_mesh}).
\end{proof}

Then we report two results from \cite{BORSUK2006, HOUSTON2017}.
\begin{lemma}[Poincar\`e inequality]  \label{lemma:poincare_on_smooth}
Let $S \subset \mathbb{R}^d$ be a bounded convex domain and let $\phi\in H^1(S)$, then
\begin{equation*}
\norm{\phi - \cfrac{1}{\abs{S}}\int_S \phi\dx}_{L^2(S)} \leq C_d \text{diam}(S)\norm{\nabla \phi}_{L^2(G)}.
\end{equation*}
\end{lemma}
\begin{lemma}[Inverse trace inequality]  \label{lemma:inverse_trace}
Let $\triangle$ be a polyhedron and $F$ be one of its faces and let
$\phi_h \in \polyspace^p(\triangle),\,\phi \in H^1(\triangle)$, then
\begin{equation*}
\begin{aligned}
\norm{\phi_h}_{L^2(F)} \leq&
C_{\text{inv}} \sqrt{\cfrac{\abs{F}}{\abs{\triangle}}}\norm{\phi_h}_{L^2(\triangle)},\\
\norm{\phi}_{L^2(F)} \leq&
C_{\text{inv}}\left(
    h_\triangle^{-1}\norm{\phi}_{L^2(\triangle)}
  + h_\triangle\seminorm{\phi}_{H^1(\triangle)}
\right).
\end{aligned}
\end{equation*}
\end{lemma}

Finally, consider the following property:
\begin{proposition}\label{prop:sumconversion}
Let $\phi \in H^1(\domain)$ and $q_h \in \V_h$, then there holds:
\begin{equation}
\begin{aligned}
\sum_{\bx_i\in \B_h} \sum_{\bx_j\in G_i}\int_{\face_{ij}} \phi \vec{n}_{ij} \PtoBox q_i\ds =&
\sum_{\face_{ij} \in \F_h} \int_{\face_{ij}} \phi \vec{n}_{ij}
\jump{\PtoBox q_h}_{ij}\ds, \\
\end{aligned}
\end{equation}
where $\jump{\PtoBox q_h}_{ij} = \PtoBox q_i - \PtoBox q_{j}$.
\end{proposition}
\begin{proof}
For any face we have two contributions from two different values of the basis functions. Passing from box summation to face summation we obtain the proof:
\begin{equation}
\begin{aligned}
\sum_{\bx_i\in \B_h} \sum_{\bx_j\in G_i}
\int_{\face_{ij}} \phi \vec{n}_{ij} \PtoBox q_i \ds=&
\sum_{\face_{ij} \in \F_h} \int_{\face_{ij}}
\phi \vec{n}_{ij} \PtoBox q_i
+ \phi \vec{n}_{ji} \PtoBox q_j\ds \\
=&\sum_{\face_{ij} \in \F_h} \int_{\face_{ij}}
\phi \vec{n}_{ij} \PtoBox q_i
- \phi \vec{n}_{ij} \PtoBox q_j\ds \\
=&\sum_{\face_{ij} \in \F_h} \int_{\face_{ij}}
\phi \vec{n}_{ij} \left(\PtoBox q_i - \PtoBox q_j\right)\ds. \\
\end{aligned}
\end{equation}
\end{proof}

\begin{delayedproof}[Proof of Proposition \ref{prop:starnorm_properties}] \label{proof:star_norm}
\begin{enumerate}
\item By properties \cite[Equations (2.11) and (2.12)]{QUARTERONI2011} for the lumping map \eqref{eq:lumping_map} and by \cite[Lemma 5.1]{COUDIERE1999}, that states that
$\forall q_h \in \V_h,\, \exists C > 0:$
\begin{equation}
\norm{\PtoBox q_h}_{L^2} \leq C \seminorm{q_h}_*.
\label{eq:star_poincare}
\end{equation}

\item $\forall q_h \in \V_h,\, \exists C > 0:$
\begin{equation*}
\begin{aligned}
\seminorm{q_h}_* =& \left(\sum_{\face_{ij} \in \F_h}
d_{ij} \int_{\face_{ij}}
\abs{\cfrac{\PtoBox q_i - \PtoBox q_j}{d_{ij}}}^2\ds\right)^{\frac{1}{2}}\\
=& \left(\sum_{\face_{ij} \in \F_h}
d\int_{\dmond_{ij}}\abs{\pdd{q_h}{}{\vec{n}_{ij}}}^2\ds\right)^{\frac{1}{2}} \leq
\sqrt{d}\seminorm{q_h}_{H^1}
\end{aligned}
\end{equation*}
where $d$ is the dimension of the space $\Real^d$ and where we used the fact that
\[\abs{\dmond_{ij}}=\abs{\cfrac{\face_{ij}d_{ij}}{d}}.\]

This holds because, for piecewise linear functions, the face-centred finite difference between box centres values coincides with the face normal gradient of the function itself.

\item $\forall q_h \in \V_h,$ by equivalence \eqref{lemma:bilforms_equivalence} and by Proposition \ref{prop:sumconversion}, it holds:
\begin{equation}
\begin{aligned}
\seminorm{q_h}_{H^1}^2 =& \int_\domain \nabla q_h \cdot \nabla q_h \dx \\
=& \sum_{\bx_i \in \B_h} \sum_{\bx_j\in G_i} \int_{\face_{ij}}
\nabla q_h \cdot \vec{n}_{ij} \PtoBox q_i \ds \\
=& \sum_{\face_{ij}\in \F_h} \int_{\face_{ij}}
\nabla q_h \cdot \vec{n}_{ij} \jump{\PtoBox q_h}_{ij} \ds. \\
\end{aligned}
\end{equation}
Now, by the Cauchy-Schwarz inequality
\begin{equation}
\begin{aligned}
\seminorm{q_h}_{H^1}^2 =&
\sum_{\face_{ij}\in \F_h} \int_{\face_{ij}}
\nabla q_h \cdot \vec{n}_{ij} \jump{\PtoBox q_h}_{ij} \ds \\
\leq& \sum_{\face_{ij}\in \F_h} \left(\int_{\face_{ij}}
\abs{\nabla q_h \cdot \vec{n}_{ij}}^2 \ds\right)^{\frac{1}{2}}
\left(\int_{\face_{ij}} \jump{\PtoBox q_h}_{ij}^2 \ds\right)^{\frac{1}{2}} \\
=& \sum_{\face_{ij}\in \F_h} d_{ij}\left(\int_{\face_{ij}}
\abs{\nabla q_h \cdot \vec{n}_{ij}}^2 \ds\right)^{\frac{1}{2}}
\left(\int_{\face_{ij}}d_{ij}
\abs{\cfrac{\jump{\PtoBox q_h}_{ij}}{d_{ij}}}^2 \ds\right)^{\frac{1}{2}} \\
\leq& \sum_{\face_{ij}\in \F_h} \left(\int_{\dmond_{ij}}
\abs{\nabla q_h}^2 \ds\right)^{\frac{1}{2}}
\left(d_{ij}\int_{\face_{ij}}
\abs{\cfrac{\jump{\PtoBox q_h}_{ij}}{d_{ij}}}^2 \ds\right)^{\frac{1}{2}} \\
\leq& \left(\sum_{\face_{ij}\in \F_h} \int_{\dmond_{ij}}
\abs{\nabla q_h}^2 \ds\right)^{\frac{1}{2}}
\left(\sum_{\face_{ij}\in \F_h} d_{ij}\int_{\face_{ij}}
\abs{\cfrac{\jump{\PtoBox q_h}_{ij}}{d_{ij}}}^2 \ds\right)^{\frac{1}{2}} \\
=&\seminorm{q_h}_{H^1}\seminorm{q_h}_*
\end{aligned}
\label{eq:star_norm:2_prop}
\end{equation}
where we used the fact that $\nabla q_h \cdot \vec{n}_{ij}$ is constant on diamond $\dmond_{ij}$ and, in the last passage, we employed the H\"older inequality. To conclude the proof divide by $\seminorm{q_h}_{H^1}$ on both sides of equation \eqref{eq:star_norm:2_prop}.

\item By Lemma \ref{lemma:poincare_on_smooth},
\begin{equation}
\begin{aligned}
\norm{q_h}_{L^2}^2 =& \sum_{\triangle\in\Tau_h}\norm{q_h}_{L^2(\triangle)}^2 \\
\lesssim& \sum_{\triangle\in\Tau_h}
\norm{q_h - \cfrac{1}{\abs{\triangle}}\int_\triangle q_h\dx}_{L^2(\triangle)}^2 + \norm{\cfrac{1}{\abs{\triangle}}\int_\triangle q_h\dx}_{L^2(\triangle)}^2 \\
\lesssim& \sum_{\triangle\in\Tau_h}
h_\triangle^2\seminorm{q_h}_{H^1(\triangle)}^2
+ \norm{\PtoBox q_h}_{L^2(\triangle)}^2 \\
\leq& h^2 \seminorm{q_h}_{H^1}^2 + \norm{\PtoBox q_h}_{L^2}^2 \\
\leq& (1 + h^2)\seminorm{q_h}_*^2.
\end{aligned}
\end{equation}

\item By triangular inequality,
\begin{equation*}
\begin{aligned}
\seminorm{q_h}_* =& \left(\sum_{\face_{ij} \in \F_h}
d_{ij}\int_{\face_{ij}}
\abs{\cfrac{\PtoBox q_i - \PtoBox q_j}{d_{ij}}}^2\ds\right)^{\frac{1}{2}}\\
\leq& \left(\sum_{\face_{ij} \in \F_h}
d_{ij}\int_{\face_{ij}}
\abs{\cfrac{\PtoBox q_i}{d_{ij}}}^2 + \abs{\cfrac{\PtoBox q_j}{d_{ij}}}^2\ds\right)^{\frac{1}{2}}\\
\leq& \cfrac{1}{d_{ij}}\left(\sum_{\face_{ij} \in \F_h}
d_{ij}\int_{\face_{ij}}
\abs{\PtoBox q_i}^2 + \abs{\PtoBox q_j}^2\ds\right)^{\frac{1}{2}}\\
\leq&\cfrac{1}{\min\limits_{\triangle \in \Tau_h} h_\triangle}\left(\sum_{\face_{ij} \in \F_h}
d_{ij}\int_{\face_{ij}}\abs{\PtoBox q_i}^2
+\abs{\PtoBox q_j}^2\ds\right)^{\frac{1}{2}} \\
\leq& 2h_m^{-1}\norm{\PtoBox q_h}_{L^2}
\end{aligned}
\end{equation*}
where $h_m = \min_\triangle h_\triangle$.
\end{enumerate}
\end{delayedproof}

\section{Appendix}  \label{app:proofs}
\begin{delayedproof}[Proof of Lemma \ref{lemma:consistency}] \label{proof:consistency}

Consider the definitions of $\C_B$ and $\widetilde{\C}_B$. By Lemma \ref{lemma:bilforms_equivalence} together with the Cauchy-Schwarz inequality and Lemmas \ref{lemma:maps_error} and \ref{lemma:poincare_on_smooth}, we get
\begin{equation}
\begin{aligned}
\C_B((\vec{u},p),&(\PtoBox \vec{v}_h,\PtoBox q_h))-
\widetilde{\C}_B((\vec{u}_B,p_B),(\PtoBox \vec{v}_h,\PtoBox q_h)) \\
=& a_B(\vec{u}-\vec{u}_B,\PtoBox \vec{v}_h)
+ b_B(\PtoBox\vec{v}_h,p) \pm b_B(\PtoBox\vec{v}_h,p_B)
- \tilde{b}_B(\PtoBox\vec{v}_h,p_B) \\
&+ c_B(\vec{u},\PtoBox q_h) \pm c_B(\PtoBox\vec{v}_h,p_B)
- \tilde{c}_B(p_B,\PtoBox q_h) \\
&+ s_B(\PtoBox q_h,p) \pm s_B(\PtoBox q_h,p_B)
- \tilde{s}_B(\PtoBox q_h,p_B) \\
=& \sum_{\triangle \in \Tau_h} \left[(\vec{f},\vec{v}_h)_\triangle
- (\vec{f},\PtoBox \vec{v}_h)_\triangle\right]
+ (b_B-\tilde{b}_B)(\PtoBox\vec{v}_h,p_B) \\
&+ (c_B-\tilde{c}_B)(\vec{u}_B, \PtoBox q_h)
+(s_B - \tilde{s}_B)(p_B,\PtoBox q_h) \\
=& \sum_{\triangle \in \Tau_h}(\vec{f}, \vec{v} - \PtoBox\vec{v}_h)_\triangle
+ (b_B-\tilde{b}_B)(\PtoBox\vec{v}_h,p_B) \\
&+ (c_B-\tilde{c}_B)(\vec{u}_B, \PtoBox q_h)
+(s_B - \tilde{s}_B)(p_B,\PtoBox q_h) \\
=& \sum_{\triangle \in \Tau_h}\left(\vec{f} - \cfrac{1}{\abs{\triangle}}\int_\triangle \vec{f}\dx,
\vec{v} - \PtoBox\vec{v}_h\right)_\triangle
+ (b_B-\tilde{b}_B)(\PtoBox\vec{v}_h,p_B) \\
&+ (c_B-\tilde{c}_B)(\vec{u}_B, \PtoBox q_h)
+(s_B - \tilde{s}_B)(p_B,\PtoBox q_h) \\
\leq& C_d h^2\norm{\vec{f}}_{L^2} \seminorm{\vec{v}_h}_{H^1}
+ (b_B-\tilde{b}_B)(\PtoBox\vec{v}_h,p_B)\\
&+ (c_B-\tilde{c}_B)(\vec{u}_B, \PtoBox q_h)
+ (s_B - \tilde{s}_B)(p_B,\PtoBox q_h).
\label{eq:consistency_0}
\end{aligned}
\end{equation}

We have now to estimate the last three terms. We first use Proposition \ref{prop:sumconversion} to sum the integrals over mesh faces. Consider the face barycentres $\faceCentre_{ij}$ and notice that, being $p_B$ and $\vec{u}_B$ piecewise linear, the following hold:
\begin{equation}
w_{ij}\PtoBox p_i + (1-w_{ij})\PtoBox p_j = p_B(\faceCentre_{ij}),
\quad\text{and}\quad
w_{ij}\PtoBox \vec{u}_i + (1-w_{ij})\PtoBox \vec{u}_j
= \vec{u}_B(\faceCentre_{ij})
\label{eq:pl_equality_on_face_centre}
\end{equation}
where we employed notation \eqref{eq:notation_ptobox}. Let now $\vec{x}$ be a point on the face $\face_{ij}$. As $p_B$ and $\vec{u}_B$ are  piecewise linear, using a Taylor expansion around $\faceCentre_{ij}$ we have:
\begin{equation}
\begin{aligned}
p_B(\vec{x}) =& p_B(\faceCentre_{ij}) + \sum_{\triangle \in \Tau_h: \triangle\cap\dmond_{ij}\neq\emptyset}
(\vec{x} - \faceCentre_{ij})\cdot\nabla p_B
\indicator{\triangle\cap\dmond_{ij}}, \\
\vec{u}_B(\vec{x}) =& \vec{u}_B(\faceCentre_{ij}) +
\sum_{\triangle \in \Tau_h: \triangle\cap\dmond_{ij}\neq\emptyset}
(\vec{x} - \faceCentre_{ij})^\intercal\nabla \vec{u}_B
\indicator{\triangle\cap\dmond_{ij}}, \\
\end{aligned}
\label{eq:pl_taylor_exp_on_dmond}
\end{equation}
where $\indicator{}$ is the indicator function and $\nabla p_B$ is piecewise constant on each intersection between triangle $\triangle$ and diamond $\dmond_{ij}$ (c.f. Figure \ref{fig:scheme_dual_geom}).

Employing now equations \eqref{eq:pl_equality_on_face_centre} and \eqref{eq:pl_taylor_exp_on_dmond} and Proposition \ref{prop:sumconversion}, we obtain that
\begin{equation}
\begin{aligned}
(b_B -\tilde{b}_B)(\PtoBox\vec{v}_h,p_B) =&
\sum_{\bx_i \in \B_h}\sum_{\bx_j\in G_i} \int_{\face_{ij}}
\left[w_{ij}\PtoBox p_i + (1-w_{ij})\PtoBox p_j - p_B\right]
\PtoBox\vec{v}_h \cdot \vec{n}_{ij}\ds\\
=& -\sum_{\face_{ij}\in \F_h} \cfrac{d_{ij}}{d_{ij}}\int_{\face_{ij}}
\left((\vec{x}-\faceCentre_{ij})\cdot\nabla p_B\right)
\left(\jump{\PtoBox\vec{v}_h}_{ij}
\cdot \vec{n}_{ij}\right) \ds\\
=& -\sum_{\face_{ij}\in \F_h} d_{ij}\int_{\face_{ij}}
\left((\vec{x}-\faceCentre_{ij})\cdot\nabla p_B\right)
\left(\pdd{\vec{v}_h}{}{\vec{n}_{ij}}
\cdot \vec{n}_{ij}\right)\ds\\
\leq& \sum_{\face_{ij}\in \F_h} \left(
    d_{ij}\int_{\face_{ij}}
    \abs{\vec{x}-\faceCentre_{ij}}^2\abs{\nabla p_B}^2\ds
\right)^{\frac{1}{2}}
\left(
    d_{ij}\int_{\face_{ij}}
    \abs{\pdd{\vec{v}_h}{}{\vec{n}_{ij}}}^2
    \abs{\vec{n}_{ij}}^2\ds
\right)^{\frac{1}{2}} \\
\leq& \sum_{\face_{ij}\in \F_h} \left(d_{ij}^2
    d_{ij}\int_{\face_{ij}}\abs{\nabla p_B}^2\ds
\right)^{\frac{1}{2}}
\left(
    d_{ij}\int_{\face_{ij}}
    \abs{\pdd{\vec{v}_h}{}{\vec{n}_{ij}}}^2\ds
\right)^{\frac{1}{2}} \\
\leq& \left(\sum_{\face_{ij}\in \F_h}  d_{ij}^2
    d_{ij}\int_{\face_{ij}}\abs{\nabla p_B}^2\ds
\right)^{\frac{1}{2}}
\left(\sum_{\face_{ij}\in \F_h}
    d_{ij}\int_{\face_{ij}}
    \abs{\pdd{\vec{v}_h}{}{\vec{n}_{ij}}}^2\ds
\right)^{\frac{1}{2}} \\
\leq&h\seminorm{p_B}_{H^1}\seminorm{\vec{v}_h}_{*}. \\
%%%%%%%%%%%%%%
(c_B -\tilde{c}_B)(\vec{u}_B,\PtoBox q_h) =&
\sum_{\bx_i \in \B_h}\sum_{\bx_j\in G_i} \int_{\face_{ij}}
\left[w_{ij}\PtoBox \vec{u}_i + (1-w_{ij})\PtoBox \vec{u}_j - \vec{u}_B\right]
\cdot\vec{n}_{ij} \PtoBox q_h \ds\\
=& \sum_{\face_{ij}\in \F_h} \int_{\face_{ij}}
(\vec{x}-\faceCentre_{ij})^\intercal \nabla\vec{u}_B
\vec{n}_{ij} \jump{\PtoBox q_h}_{ij} \ds\\
=& \sum_{\face_{ij}\in \F_h} \int_{\face_{ij}}
(\vec{x}-\faceCentre_{ij}) \cdot
\pdd{\vec{u}_B}{}{\vec{n}_{ij}}
\jump{\PtoBox q_h}_{ij} \ds\\
=& \sum_{\face_{ij}\in \F_h}
\jump{\PtoBox q_h}_{ij}
\pdd{\vec{u}_B}{}{\vec{n}_{ij}}\cdot
\int_{\face_{ij}}(\vec{x}-\faceCentre_{ij})\ds=0,\\
\end{aligned}
\label{eq:interp_discr_consistency}
\end{equation}

For what concerns the last term in inequality \eqref{eq:consistency_0}, we resort to estimate \eqref{eq:rhie_chow_consistency}. This concludes the proof.
\end{delayedproof}

\begin{delayedproof}[Proof of Lemma \ref{lemma:boundedness}] \label{proof:boundedness}

Write the compact form $\C_B$:
\begin{equation}
\begin{aligned}
\C_B((\vec{v}, q),(\PtoBox \vec{v}_h, \PtoBox q_h)) =&
a_B(\vec{v},\PtoBox \vec{v}_h) +
b_B(\PtoBox\vec{v}_h,q) +
c_B(\vec{v},\PtoBox q_h) +
s_B(q,\PtoBox q_h)\\
=& (I) + (II) + (III) + (IV).
\end{aligned}
\end{equation}
Consider now each term separately. Employing Lemma \ref{lemma:bilforms_equivalence} and Proposition \ref{prop:starnorm_properties}. We have
\begin{equation*}
(I) = a(\vec{v},\vec{v}_h) \leq \nu \seminorm{\vec{v}}_{H^1}\seminorm{\vec{v}_h}_{H^1}
\leq \nu \seminorm{\vec{v}}_{H^1}\seminorm{\vec{v}_h}_{*}.
\end{equation*}

Knowing that $\norm{\nabla\cdot\vec{v}}_{L^2}\leq \sqrt{d}\norm{\nabla\vec{v}}_{L^2}\,\forall\vec{v}\in\bs{\V}$, by the Cauchy-Schwarz inequality, we obtain
\begin{equation*}
\begin{aligned}
(II)=&b(\vec{v}_h,q)\leq
\sqrt{d} \seminorm{\vec{v}_h}_{H^1}\norm{q}_{L^2}\leq
\sqrt{d} \seminorm{\vec{v}_h}_{*}\norm{q}_{L^2}, \\
(III)=&\sum_{\bx_i \in \B_h}\int_{\bx_i}
\nabla\cdot\vec{v}\PtoBox q_h\dx \\
\leq& \sqrt{d} \sum_{\bx_i \in \B_h}
\seminorm{\vec{v}}_{H^1(\bx_i)}\norm{\PtoBox q_h}_{L^2(\bx_i)}
\leq \sqrt{d} \seminorm{\vec{v}}_{H^1}\norm{\PtoBox q_h}_{L^2},
\end{aligned}
\end{equation*}
where we used the H\"older inequality in the last step.

For the Rhie-Chow stabilization we recall inequality \eqref{eq:rhie_chow_continuity_cont} to have
\begin{equation*}
(IV)\lesssim h^{\frac{3}{2}}
\seminorm{\nabla q}_{h,1}\seminorm{q_h}_{\triangle,*}.
\end{equation*}

Combining the above estimates, we have
\begin{equation*}
\begin{aligned}
\C_B((\vec{v},q),(\PtoBox \vec{v}_h,\PtoBox q_h))
\lesssim& \seminorm{\vec{v}}_{H^1}\seminorm{\vec{v}_h}_{*}
+ \seminorm{\vec{v}_h}_{*}\norm{q}_{L^2}
+ \seminorm{\vec{v}}_{*}\norm{\PtoBox q_h}_{L^2}
+ h^{\frac{3}{2}} \seminorm{\nabla q}_{h,1}\seminorm{q_h}_{\triangle,*} \\
\leq& \left(\seminorm{\vec{v}}_{H^1}
+ \norm{q}_{L^2} + h^{\frac{3}{2}} \seminorm{\nabla q}_{h,1}\right)\trinorm{\vec{v}_h,q_h}_{box}.
\end{aligned}
\end{equation*}

\end{delayedproof}

\begin{delayedproof}[Proof of Lemma \ref{lemma:boundedness_tilde}] \label{proof:boundedness_tilde}

We recall that
\begin{equation*}
\begin{aligned}
\widetilde{\C}_B((\vec{v}_h,q_h),(\PtoBox \vec{w}_h,\PtoBox z_h)) =&
a_B(\vec{v}_h,\PtoBox \vec{w}_h) +
\tilde{b}_B(\PtoBox\vec{w}_h,q_h) +
\tilde{c}_B(\vec{v}_h,\PtoBox z_h) +
\tilde{s}_B(q_h,\PtoBox z_h).
\end{aligned}
\end{equation*}
The estimates of the terms $a_B$ ans $s_B$ are obtained as in the proof of Lemma \ref{lemma:boundedness} (c.f. proof \ref{proof:boundedness}). On the other hand, for $\tilde{b}_B$ and $\tilde{c}_B$ we employ equation \eqref{eq:interp_discr_consistency} (from the proof of Lemma \ref{lemma:consistency}), and equation \eqref{eq:interp_discr_consistency}:
\begin{equation}
\begin{aligned}
\tilde{b}_B(\PtoBox\vec{w}_h,q_h)
\pm b_B(\PtoBox\vec{w}_h,q_h) =&
 b_B(\PtoBox\vec{w}_h,q_h) + (\tilde{b}_B - b_B)(\PtoBox\vec{w}_h, q_h)\\
\lesssim& \sqrt{d} \seminorm{\vec{w}_h}_{*}\norm{q_h}_{L^2}
+ h\seminorm{q_h}_{H^1}\seminorm{\vec{w}_h}_{\triangle,*}, \\
\tilde{c}_B(\vec{v}_h,\PtoBox z_h)
\pm c_B(\vec{v}_h,\PtoBox z_h) =&
 c_B(\vec{v}_h,\PtoBox z_h) + (\tilde{c}_B - c_B)(\vec{v}_h,\PtoBox z_h)\\
\lesssim& \sqrt{d} \seminorm{\vec{v}_h}_{*}\norm{\PtoBox z_h}_{L^2}. \\
\end{aligned}
\end{equation}

By equation \eqref{eq:rhie_chow_continuous_dicrete} and gathering the above estimates the proof is completed.
\end{delayedproof}

\begin{delayedproof}[Partial proof of Lemma \ref{lemma:coercivity}] \label{proof:coercivity}
Let us first notice that, using Proposition \ref{prop:sumconversion} and applying relationship \eqref{eq:pl_equality_on_face_centre} to $q_h$ and $\vec{v}_h$,
\begin{equation}
\begin{aligned}
\tilde{b}_B&(\PtoBox\vec{v}_h,q_h) + \tilde{c}_B(\vec{v}_h,\PtoBox q_h)=\\
=& \sum_{\face_{ij}\in \F_h}\int_{\face_{ij}}
(w_{ij}\PtoBox q_i + (1-w_{ij})\PtoBox q_j)\vec{n}_{ij}
\cdot\jump{\PtoBox \vec{v}_h}_{ij}\ds \\
&+\int_{\face_{ij}}
(w_{ij}\PtoBox \vec{v}_i + (1-w_{ij})\PtoBox \vec{v}_j)\cdot\vec{n}_{ij}
\jump{\PtoBox q_h}_{ij}\ds \\
=& \sum_{\face_{ij}\in \F_h} d_{ij} \int_{\face_{ij}}
\left[
    q_h \vec{n}_{ij} \cdot \pdd{\vec{v}_h}{}{\vec{n}_{ij}}
  + \vec{v}_h \cdot \vec{n}_{ij}
    \pdd{q_h}{}{\vec{n}_{ij}}
\right]\bigg\vert_{\vec{f}_{ij}}\ds \\
=& \sum_{\face_{ij}\in \F_h} d_{ij} \int_{\face_{ij}}
\pdd{q_h \vec{v}_h}{}{\vec{n}_{ij}}
\bigg\vert_{\vec{f}_{ij}}\cdot \vec{n}_{ij}\ds \\
=& \sum_{\face_{ij}\in \F_h}\sum_{\bx_j\in G_i} \int_{\face_{ij}}
\jump{\PtoBox q_h \PtoBox\vec{v}_h}_{ij}\cdot\vec{n}_{ij}\ds
= \sum_{\bx_i \in \B_h}\PtoBox q_i \PtoBox\vec{v}_i
\sum_{\bx_j\in G_i}\int_{\face_{ij}}\vec{1}\cdot\vec{n}_{ij}\ds \\
=& \sum_{\bx_i \in \B_h}
\PtoBox q_i \PtoBox\vec{v}_i
\int_{\bx_i}\nabla\cdot\vec{1}\dx = 0.\\
\end{aligned}
\label{eq:cbt_bbt_equivalence}
\end{equation}

Then, using Lemma \ref{lemma:bilforms_equivalence}, we have
\begin{equation*}
\begin{aligned}
\widetilde{\C}_B((\vec{v}_h,q_h),(\PtoBox \vec{v}_h,\PtoBox q_h))=&
a_B(\vec{v}_h,\PtoBox \vec{v}_h) +
\tilde{b}_B(\PtoBox\vec{v}_h,q_h) +
\tilde{c}_B(\vec{v}_h,\PtoBox q_h) +
\tilde{s}_B(q_h,\PtoBox q_h) \\
=& a(\vec{v}_h,\vec{v}_h) + \tilde{s}_B(q_h,\PtoBox q_h) \\
=& \nu\seminorm{\vec{v}_h}_{H^1}^2 + \tilde{s}_B(q_h,\PtoBox q_h).
\end{aligned}
\end{equation*}

Employing now Conjecture \ref{lemma:rhiechow_coercivity} concludes the proof.
\end{delayedproof}

\bibliographystyle{plain}
\bibliography{bibliography}

\end{document}